\documentclass[11pt]{article}
\usepackage[margin=2cm]{geometry}
\usepackage{amsmath}
\usepackage{graphicx}
\usepackage{amssymb}
\usepackage{wasysym}
\usepackage{gensymb}
\usepackage{wrapfig}
\usepackage{multicol}
\usepackage{xcolor}
\usepackage{caption}
\usepackage{subcaption}
\usepackage{titlesec}
\usepackage{blindtext}
\usepackage{caption}
\usepackage{enumerate}
\usepackage{amsthm}
\usepackage{url}

\titleformat{\section}
{\normalfont\large\bfseries}{\thesection}{1em}{}
\titleformat{\subsection}
{\normalfont\normalsize\bfseries}{\thesubsection}{1em}{}
\titleformat{\subsubsection}
{\normalfont\small\scshape}{\thesubsubsection}{1em}{}
\newtheorem{theorem}{Theorem}[section]
\newtheorem{prop}[theorem]{Proposition}
\newtheorem{lemma}[theorem]{Lemma}
\theoremstyle{remark}
\newtheorem{remark}{Remark}[section]
\theoremstyle{definition}
\newtheorem{definition}{Definition}[section]
\numberwithin{equation}{section}

\title{\vspace{-1.0cm} Semi-discrete linear hyperbolic  polyharmonic flows of closed polygons}
\author{\vspace{-2.0cm}James McCoy\footnote{James.McCoy@newcastle.edu.au, ORCID 0000-0001-6053-5144}, Jahne Meyer\footnote{Jahne.Meyer@newcastle.edu.au, ORCID 0009-0001-2850-8848, School of Information and Physical Sciences, University of Newcastle, Australia.  This work was completed while the first author was supported by Discovery Project DP180100431 of the Australian Research Council and the second author was supported by an UNRS Central Scholarship.  Parts of this work was completed while the first author was visiting Nanjing University of Information Science and Technology and the MATRIX Institute.  The authors are grateful for this support.}}
\date{}

\begin{document}
\maketitle
\begin{abstract}
We consider the damped hyperbolic motion of polygons by a linear semi-discrete analogue of polyharmonic curve diffusion.  We show that such flows may transition any polygon to any other polygon, reminiscent of the Yau problem of evolving one curve to another by a curvature flow, before converging exponentially to a point that, under appropriate rescaling, is a planar basis polygon. We also consider a hyperbolic linear semi-discrete flow of the Yau curvature difference flow, where a polygonal curve is able to flow to any other such that we get convergence to the target polygon in infinite time. 
\end{abstract}

\section{Introduction}
In \cite{chow2007semidiscrete}, Chow and Glickenstein considered a semi-discrete analogue of the second order curve shortening flow for smooth, closed curves.
Involving only a first time derivative, this may be considered a parabolic-type evolution.  Recently we considered the polyharmonic analogue of this work \cite{MM24}.  There we also considered a discrete analogue of the Yau problem of evolving one curve to another by a curvature flow.  In this article we replace the first time derivative by the second time derivative plus a first time derivative `damping term', thus considering a hyperbolic analogue of our previous work.  Our ensuing evolution equation now being second order in time requires more conditions in order to have a unique solution.  We prescribe not only initial conditions but also conditions at a fixed later time, thus generating flows that transition from one polygonal curve through another and then asymptotically to a limiting shape.  These flows thus flow one curve to another in the spirit of the Yau problem and then flow to a limiting polygonal curve.

Our setup similar is as in the earlier works.  Given an ordered collection of $n$ points in $\mathbb{R}^p$, joined in order to form a piecewise linear, closed immersed `curve', the flow is given by a second order system of $n$ coupled linear ordinary differential equations.  The constant coefficient matrix of this system is an $m$th order difference-type operator.  

Linearity permits analysis via finite Fourier series.  As in \cite{MM24} and in the case that $\beta>0$, we find that the higher the order $m$, the faster the convergence, in the first case under appropriate rescaling, to a regular basis polygon.  To our knowledge semi-discrete linear hyperbolic flows have not been considered before, but the hyperbolic analogue, with damping, of the curve shortening flow has been considered (see \cite{DN23} and the references contained therein).

Fully discrete flows have been considered by Rademacher and Rademacher \cite{rademacher2017solitons, rr21}.  These are related to discretisation of  partial differential equations in space and time using finite differences, while  semi-discrete flows are related to the `method of lines' where there is discretisation in all but one dimension (usually time), reducing a given system of partial differential equations to a system of ordinary differential equation.  The equations we consider are closely related to discretisations via finite differences or finite elements; a particular hyperbolic reference is \cite{DN23}.  We would also like to mention that first order semi-discrete flows have been considered in several engineering applications \cite{GK10, HM15, HM21} and we suspect corresponding second order flows to find useful applications there as well.  The idea of replacing a first time derivative by a second derivative with first derivative damping term also appears in the heavy ball with friction system \cite{GM09}.

Settings in which curvature flows result in linear ordinary or partial differential equations are quite rare.  Apart from the present setting, the main other setting where flow equations are linear is where evolving convex curves and hypersurfaces are expressed via the Gauss map parametrisation and flow speeds have a specific structure.  The linear structure was exploited by Smoczyk for curves and hypersurfaces evolving by parabolic equations in \cite{SM05}; the curve results were extended by the first author, Schrader and Wheeler in \cite{MSW23} to higher order equations and curves with general nonzero rotation number.  The first author and Otuf considered linear hyperbolic flows of curves in \cite{MO24}, while the first author considered convex hypersurfaces evolving by analogous second order linear hyperbolic flows in \cite{M24}.

The structure of this article is as follows.  In Section \ref{S:background} we set up the necessary notation and background for us to be able to state and prove our main results, including key properties of circulant matrices and the differential operators on which our flows are based.  In Section \ref{sec:prelim} we introduce the semi-discrete linear hyperbolic polyharmonic flows and uncover some fundamental properties of these flows.  In Section \ref{S:planar} we specialise to evolving planar polygons, finding a representation formula for general solutions for the flow that allows us to deduce long-time behaviour. In Section \ref{sec:higher_codimension} we consider evolving polygonal curves in general codimension.
Finally in Section \ref{S:Yau} we consider a nonhomogeneous version of the flow which addresses the discrete analogue of Yau's curvature difference flow.

The authors are grateful to Professor Ben Andrews for a discussion on self-similar solutions in higher codimension.

\section{The semi-discrete linear hyperbolic polyharmonic flow of polygons} \label{S:background}

\begin{definition}\label{def:polygon}
An $n$ sided polygon, or $n-$gon, $\vec{X}$, is defined as a set of points $\vec{X}=(X_0, X_1, \ldots, X_{n-1})^T.$ For $j=0,1,\ldots, n-1,$ each vertex $X_j\in \mathbb{R}^p$ for some integer $p\geq 2.$
We set $X_n = X_0$ so the polygon is closed and the indices of the vertices are considered modulo $n.$
\end{definition}

The \emph{vertices} of the polygon are the points $X_j$, $j = 0,1,,\ldots, n-1$. The \emph{edges} of the polygon are the line segments $\overline{X_jX_{j+1}}$  joining each vertex $X_j$ to $X_{j+1}$. 
In general, $\vec{X}$ is a $n\times p$ matrix with real entries. When $p=2$, the polygon lies in the plane.  In this case we may consider each vertex, $X_j = (x_j\  y_j) \in \mathbb{R}^2,$ to be a point in $\mathbb{C}$, writing $ X_j = x_j + iy_j.$  Thus $\vec{X}$ as an $n \times 2$ matrix representing a polygon in $\mathbb{R}^2$ can also be thought of as a vector in $\mathbb{C}^n$.

A \textit{normal} to each vertex $X_j$ is given by
\begin{equation} \label{E:Nj}
N_j = (X_{j+1} - X_j) + (X_{j-1} - X_j), 
\end{equation}
such that the corresponding system of equations for each normal vector can be expressed in matrix form as 
\begin{equation}
    \vec{N} = M\vec{X}
\end{equation}
where $M$ is the $n \times n$ matrix given by
\begin{equation}\label{matrix:M}
M = 
\begin{bmatrix}
-2 & 1 & 0 & \cdots & 0 & 1\\
1 & -2 & 1 & 0 & \vdots & 0 \\
0 & 1 & -2 & 1 & 0 & \vdots\\
\vdots & 0 & \ddots & \ddots & \ddots & 0\\
0 & \ddots & 0 & 1 & -2 & 1\\
1 & 0 & \cdots & 0 & 1 & -2\\
\end{bmatrix}.\\
\end{equation}

\begin{definition}
Let $\vec{X}(t)$ be a family of polygons as given in Definition \ref{def:polygon}. Given $m\in \mathbb{N}$ and a constant $\beta\geq 0$, polygons $\vec{X}(t)$ satisfying
\begin{equation}\label{eqn:polyflow}
\frac{d^2\vec{X}}{dt^2} + \beta \frac{d\vec{X}}{dt} = (-1)^{m+1}M^m\vec{X} \tag{$SHPF_m$}
\end{equation}
 evolve by the $2m$th order \textit{semi-discrete linear hyperbolic polyharmonic flow}, where $M$ is the matrix given in (\ref{matrix:M}). 
 \end{definition}

\begin{remark}
\begin{enumerate} 
  \item The elements of $M$ in \eqref{matrix:M} are the coefficients in finite difference approximations for the `second spatial derivatives' associated $\vec{X}$, as in \eqref{E:Nj}.  Powers of $M$ then yield higher finite differences, consistent with those obtained for example by Newton's divided difference approach.\\

  \item We have restricted to $\beta\geq 0$ as we wish to allow only no damping ($\beta=0$) or damping in the traditional sense of a mechanical or electrical system.  One could also consider $\beta<0$ above but this is not relevant for our applications.\\

  \item The case $m=0$ of \eqref{eqn:polyflow} is $$\frac{d^2\vec{X}}{dt^2} + \beta \frac{d\vec{X}}{dt} = - \vec{X}\mbox{.}$$ Solution behaviour will depend on the value of $\beta$.  If, for example, $\beta=0$ the general solution is 
  \begin{equation} \label{E:0}
    \vec{X}(t) = \left( \cos t \right) \vec{Y} + \left( \sin t \right) \vec{Z} 
    \end{equation}
    for arbitrary polygons $\vec{Y}$ and $\vec{X}$.  
  Polygons $\vec{Y}$ and $\vec{Z}$ can then be determined for given initial data (or more general data).  In particular, for this $\beta$ we observe that solutions are \emph{breathers}, that is, periodic in time.  They are also \emph{ancient} (can be extended back to $t\rightarrow -\infty$) and \emph{eternal} (exist for all time).
As a specific example, the solution with $\vec{X}(0)=X_0$ and $\vec{X}'(0) = 0$ (the `trivial' polygon with all points at the origin) is 
$$\vec{X}(t) = \left( \cos t\right) X_0 \mbox{.}$$
  On the other hand, if we restrict the general solution \eqref{E:0} to the time interval $\left[ 0, \frac{\pi}{2} \right]$ we have a transition from $\vec{X}(0)= \vec{Y}$ to $\vec{X}\left( \frac{\pi}{2}\right) = \vec{Z}$, a kind of finite, exact solution to the Yau problem in this setting using a second order flow.  Extending the time interval, the solution is breathing between states $\vec{Y}$ and $\vec{Z}$.

  Alternatively, for $0<\beta<2$, the general solution has the form 
  $$\vec{X}(t) = e^{-\frac{\beta}{2}t} \left[ \cos \left( \gamma \, t\right) \vec{Y} + \sin \left( \gamma \, t\right) \vec{Z} \right]$$
  where $\gamma = \sqrt{1-\left( \frac{\beta}{2}\right)^2 }$ and we observe solutions again exist for all time but $\vec{X}\left( t\right) \rightarrow 0$ as $t\rightarrow \infty$.  The rescaled solutions $e^{\frac{\beta}{2}} \vec{X}\left( t\right)$ are exactly 
  $$e^{\frac{\beta}{2}t}\vec{X}(t) =  \cos \left( \gamma t\right) \vec{Y} + \sin \left( \gamma t\right) \vec{Z} \mbox{.}$$
  Thus the solutions exhibit asymptotic breathing behaviour as $t\rightarrow \infty$ and under the same rescaling the ancient solution emanates from an asymptotically breathing solution between the states $\vec{Y}$ and $\vec{Z}$.

  In the case $\beta = 2$ the solution has the form
  $$\vec{X}\left( t\right) = \left( 1-t\right) e^{-t} \,\vec{Y} + t\, e^{\left( 1-t\right)} \vec{Z} \mbox{.}$$
  Again solutions make sense for all time and $\vec{X}\left( t\right) \rightarrow 0$ as $t\rightarrow \infty$.  Observe that rescaled solutions satisfy
  $$\frac{1}{t} e^t \vec{X}\left( t\right) = \left( \frac{1}{t} - 1\right) \vec{Y} + e\, \vec{Z} \rightarrow -\vec{Y} + e\, \vec{Z} \mbox{,}$$
  that is, under this rescaling they approach a precise linear combination of the given data $\vec{Y}$ and $\vec{Z}$.  One can argue similarly in considering `where ancient solutions come from'.
  
  Finally for $\beta>2$ the solution has the form
  $$\vec{X}(t) = e^{r_- t} \vec{Y} + e^{r_+ t} \vec{Z}$$
  where $r_\pm = -\frac{\beta}{2} \pm \sqrt{\left(\frac{\beta}{2}\right)^2 -1}$ are both negative.  So in this case solutions again decay to zero but there are no oscillations.  Rescaling we observe in this case
  $$e^{-r_+ t} \vec{X}(t) \rightarrow \vec{Z} \mbox{.}$$
  
  For the remainder of this article we consider $m\in \mathbb{N},$ the strictly positive integers.\\

  \item The system of ODEs \eqref{eqn:polyflow} is second order in time and a natural approach to working with it would be to set $\vec{Y} = \frac{d \vec{X}}{dt}$ and consider the associated first order system of $2n$ equations.\\
  
  \item As \eqref{eqn:polyflow} is a homogeneous system of linear differential equations with a constant coefficient matrix, existence of a unique solution in a neighbourhood of any initial data is completely standard as one can write down an explicit formula for the solution.  Moreover the formula for the solution reveals that the solution exists for all time, given any initial polygon and indeed ancient solutions also make sense for any initial polygon.  These are properties that usually do not hold in full generality for other geometric flows and generally require much work to prove, whether they hold in general or under particular conditions.

  \end{enumerate}

  \end{remark}

\section{Preliminary properties of the semi-discrete hyperbolic 
 polyharmonic flow}\label{sec:prelim}

In this section we detail the setup of the semi-discrete linear hyperbolic polyharmonic flow and discuss some of its properties.  Some of this material is similarly described in \cite{chow2007semidiscrete} and \cite{MM24}.

The matrix $M$ given in (\ref{matrix:M}) is a \textit{circulant} matrix. 
A general circulant matrix $B$ has the form
\begin{equation*}
B = 
\begin{bmatrix}
b_0 & b_1 & b_2 & \cdots & b_{n-2} & b_{n-1}\\
b_{n-1} & b_0 & b_1  & \cdots& b_{n-3} & b_{n-2} \\
b_{n-2} & b_{n-1} & b_0  & \cdots & b_{n-4} & b_{n-3}\\
\vdots & \vdots & \vdots & \cdots & \vdots & \vdots\\
b_2 & b_3 & b_4 &\cdots & b_0 & b_1\\
b_1 & b_2 & b_3 & \cdots  & b_{n-1} & b_0\\
\end{bmatrix}
\end{equation*}
where each row is produced by shifting each of the elements of the previous row to the right. A matrix of this form is also denoted $B = \text{circ}(b_0, b_1,\ldots, b_{n-1}).$
Many properties of circulant matrices are detailed in \cite{davismatrices} with further details about their eigenvalues and eigenvectors in \cite{tee2007eigenvectors}. The product of circulant matrices is circulant, therefore $M^m$ is circulant for any $m \in \mathbb{N}.$ For a characterisation of the elements of $M^m$ we refer to \cite[Lemmas 3.2, 3.3]{MM24}.

Our next result is an extension of \cite[Lemma 3.5]{MM24} to include the second time derivative.

\begin{lemma}\label{lem:matrix_properties}
Let a vector in $\mathbb{R}^n$ with all entries equal to the same constant $c$ be denoted by $\vec{c}$. Also let $E$ denote a $p \times p$ matrix. We have the following properties:
\begin{enumerate}
    \item $M^m\vec{1} = \vec{0}$, and for any $a_i \in \mathbb{R}$ for $i=1, 2, \ldots, p$,
    \item $\frac{d}{dt}\left(\vec{X}E + (\vec{a}_1, \vec{a}_2, \ldots \vec{a}_p)\right) = (-1)^{m+1}M^m(\vec{X}E + (\vec{a}_1, \vec{a}_2, \ldots \vec{a}_p))$ and
    \item $\frac{d^2}{dt^2}(\vec{X}E + (\vec{a}_1, \vec{a}_2, \ldots \vec{a}_p)) = M^{2m}(\vec{X}E + (\vec{a}_1, \vec{a}_2, \ldots \vec{a}_p))$.
\end{enumerate}
\end{lemma}

\begin{proof}
    The first two parts above are proved in \cite{MM24}.  The third follows by a simple calculation.
    \end{proof}

\begin{remark}
\begin{enumerate}
   \item As remarked in \cite{MM24}, the polygon $\vec{X}$ can be considered as a graph $G=\{V, E\}$ where $V= \{X_0,\ldots X_{n-1}\}$ are the vertices of the polygon, and $E$ is the set of edges between consecutive vertices and the degree of each vertex is 2. Therefore the matrix $L= -M$ is a Laplacian matrix which has a corresponding definition as a Laplacian operator \cite{chung1997spectral, wardetzky2007discrete}. We have $-L^m = (-1)^{m+1}M^m$ and the semi-discrete polyharmonic flow can be associated with a higher order linear hyperbolic flow for graphs.
\end{enumerate}
    
\end{remark}

\subsection{Matrix properties}

The eigenvectors of any $n\times n$ circulant matrix are in $\mathbb{C}^n$ and  given by,
\begin{equation}\label{eqn:epolygons}P_k = (1, \omega^k, \omega^{2k}, \ldots, \omega^{(n-1)k})^T, \text{ for } k = 0, 1, \ldots , n-1,
\end{equation}
where $\omega = e^{2\pi i/n}.$ 
Here powers of $\omega$ are the $n$th roots of unity.  Each $P_k$ may be thought of as a polygon by placing the entries of $P_k$ into $\mathbb{C}$ as the vertices of the polygon, and joining consecutive entries by arrows. In particular, $P_0 = (1,1,\ldots,1)^T$ is a point, and the remaining $P_k$ are either regular convex polygons or star shaped regular polygons. Each pair $P_k$ and $P_{n-k}$ comprises the same regular polygon, but with the arrows in the opposite orientation. The exception is when $n$ is even then $P_{\frac{n}{2}}$ in isolation is a  line interval where the arrows between points overlap each other $\frac{n}{2}$ times. For each $n$, the collection $\{P_k\}_{k=1}^{n-1}$ produces all regular polygons whose number of sides is a factor of $n$ with $P_{\frac{n}{2}}$ as the exception for even $n$.

\begin{prop}\label{prop:chow}\cite{chow2007semidiscrete}
The set of eigenvectors, $\{\frac{1}{\sqrt{n}}P_k\}^{n-1}_{k=0}$ of $M$ forms an orthonormal basis for $\mathbb{C}^n$ where the corresponding eigenvalue, $\lambda_k,$ for each $P_k$ is given by \begin{equation*} 
  \lambda_k = -4\sin^2\left(\frac{\pi k}{n}\right).
\end{equation*}
\end{prop}

Importantly for us, denoting by $\lambda_{m,k}$ the eigenvalues of matrix $(-1)^{m+1}M^m$, with corresponding to eigenvectors $P_k$, we have 
\begin{equation}\label{eqn:eigenvalues}
    \lambda_{m,k} := (-1)^{m+1}\lambda_k^m = - 4^m \left[ \sin^2\left(\frac{\pi k}{n}\right) \right]^m.
\end{equation}
Observe that all eigenvalues are negative, except for $\lambda_{m, 0}$.

The eigenvalues of a circulant matrix occur in conjugate pairs with some exceptions. That is, $\lambda_k = \overline{\lambda_{n-k}}$ with the exception of $\lambda_0$ and $\lambda_{\frac{n}{2}}$ for $n$ even (\cite{davismatrices},\cite{tee2007eigenvectors}). Furthermore, if the circulant matrix is real and symmetric, the eigenvalues are real and thus each eigenvalue pair $\lambda_k = \lambda_{n-k}$.  This holds for $(-1)^{m+1}M^m$ for each $m$ and so we have $\lambda_{m,0} = 0, \lambda_{m,k} = \lambda_{m,n-k}$ and $\lambda_{m, n/2}$ does not have an equal pair for $n$ even. Also $\lambda_{m,k} < \lambda_{m,1} < 0$ for all $k = 2,\ldots ,\lfloor \frac{n}{2} \rfloor.$

\begin{lemma}\label{lem:constant_vec}
Given a vector $\vec{x}\in \mathbb R^n,$ if  
\begin{equation}\label{eqn:constant_vec_eqn}
    (-1)^{m+1}M^m\vec{x} = \vec{c} 
\end{equation}

for a constant $c,$ then $c=0.$
\end{lemma}

\begin{proof}
    We use the diagonalisation of the matrix $(-1)^{m+1}M^m$ which is given by
\begin{equation}\label{eqn:diagnoalised}
    (-1)^{m+1}M^m = \frac{1}{n}F\text{diag}\left(\lambda_{m,k}\right)\bar{F},
\end{equation}

where $F$ is the Fourier matrix,
\begin{equation}\label{eqn:Fouriermatrix}
    F = \left[P_0 \bigg| P_1 \bigg| \cdots \bigg| P_{n-1}\right] = 
    \begin{bmatrix}
1 & 1 & 1 & \cdots  & 1\\
1 & \omega^1 & \omega^2  & \cdots & \omega^{n-1} \\
1 & \omega^2 & \omega^4 & \cdots  & \omega^{2(n-1)}\\
\vdots & \vdots & \vdots & \ddots & \vdots\\
1 & \omega^{n-1} & \omega^{2(n-1)} & \cdots & \omega^{(n-1)(n-1)}\\
\end{bmatrix},
\end{equation}
and $\text{diag}(\lambda_{m,k})$ is the diagonal matrix with diagonal entries given by the eigenvalues $\lambda_{m,0}, \lambda_{m,1},\ldots, \lambda_{m,n-1}.$ The matrix $\bar{F}$ in (\ref{eqn:diagnoalised}) is the complex conjugate of the matrix $F$, where $F^{-1} = \frac{1}{n}\bar{F}.$

Rearranging (\ref{eqn:constant_vec_eqn}) gives
$$\text{diag}(\lambda_{m,k})\bar{F}\vec{x} = \bar{F}\vec{c}.$$

Considering the first entry of each side in this equation gives $\lambda_{m,0}\sum_{j=1}^nx_j = nc$ and, given $\lambda_{m,0} = 0$, we conclude $c=0.$
\end{proof}

We complete this section with a restatement of \cite[Lemma 3.7]{MM24} which will be needed in next section.  The result characterises the solutions $\vec{x}$ in Lemma \ref{lem:constant_vec} as just the  vectors with all entries equal.

\begin{lemma}\label{lem:nullspace_matrix}
Given a vector $\vec{x}\in\mathbb{R}^n,$ if $(-1)^{m+1}M^m\vec{x} = \vec{0}$ then $\vec{x}$ is a constant vector with all entries equal. That is  $\vec{x} = \vec{c} = \left( c, c, \ldots, c\right)^T$ for a constant $c$. 
\end{lemma}

\section{Planar solutions to the semi-discrete linear hyperbolic polyharmonic flow} \label{S:planar}

In this section we consider planar solutions of \eqref{eqn:polyflow}.  First, in part \ref{sec:selfsim_solutions}, we consider self similar solutions, then, in part \ref{sec:flow_solutions} we consider solutions with general initial polygon $\vec{X}(0) = \vec{X}_0.$  Because \eqref{eqn:polyflow} is second order in time, specifying only $\vec{X}(0)= \vec{X}_0$ does not yield a unique solution in general.

\subsection{Planar self-similar solutions}\label{sec:selfsim_solutions}
Here we are interested in \emph{self-similar solutions} to \eqref{eqn:polyflow}, that is, solutions $\vec{X}(t)$ related to the initial polygon $\vec{X}(0) = \vec{X}_0$ via the formula

\begin{equation} \label{E:ss}
\vec{X}(t) =  g(t)\vec{X}_0R(f(t)) + \vec{h}(t),
\end{equation} 
where $g,f :\mathbb{R} \rightarrow\mathbb{R}$  and $\vec{h}: \mathbb{R} \rightarrow \mathbb{M}_{n\times 2}\left( \mathbb{R}\right)$ are twice-differentiable functions.  Here $g(t)$ represents scaling ($g(t) > 0$ for all $t\in \left[ 0, T\right)$), $R(f(t))$ is the $2\times 2$ rotation matrix by angle  $f(t),$ and $\vec{h}(t)$ is a $n\times 2$ matrix corresponding to translation where $\vec{h}(t) = (\vec{h}_1(t)\ \vec{h}_2(t))$ and $h_1, h_2: \mathbb{R} \rightarrow \mathbb{R}$. We have $g(0) = 1,\  f(0) = 0,$ and $\vec{h}(0) = 0_{n\times 2},$ where $0_{n\times 2}$ is the $n\times 2$ zero matrix, such that $\vec{X}(0) = \vec{X}_0.$ Note that when considering $\vec{X}$ as a polygon in $\mathbb{C},$ equation \eqref{E:ss} is instead written
$$\vec{X}(t) = g(t) e^{i \, f(t)} \vec{X}_0 + \vec{h}(t).$$
where rotation is given by $e^{if(t)}$ and translation by $\vec{h}(t) \in \mathbb{C}^n.$

\begin{prop}\label{prop:selfsimilar}
If a family of polygons in the plane $\vec{X}(t)$ is a self-similar solution to the flow (\ref{eqn:polyflow}) by scaling, then $\vec{X}(t)$ has the form 

\begin{equation}
      \vec{X}(t) = g\left( t\right)(c_1 P_k + c_2 P_{n-k})
\end{equation}
for some fixed $k\in \left\{ 0, 1, \ldots, \lfloor \frac{n}{2} \rfloor \right\}$, where
$$g\left( t\right) = \begin{cases}
    c\, t + 1 & \mbox{ for }\beta=0 \mbox{ and } k=0\\
    \frac{c}{\beta} + \left[ 1 - \frac{c}{\beta} \right] e^{-\beta\, t} & \mbox{ for } \beta > 0 \mbox{ and } k =0\\
 e^{-\frac{\beta}{2}t} \left[ \cos \left( \gamma_{m, k}\, t\right) +   
    c\, \sin\left( \gamma_{m, k}\, t\right) \right] & \mbox{ for } 0 \leq \beta < 2 \sqrt{\left| \lambda_{m, k} \right|} \mbox{ and } k\in \left\{ 1, \ldots, \lfloor \frac{n}{2} \rfloor \right\} \mbox{.}\\
    e^{r_+ t} + c\, \left( e^{r_-t} - e^{r_+ t} \right) & \mbox{ for } \beta \geq 2 \sqrt{\left| \lambda_{m, k} \right|} \mbox{ and } k\in \left\{ 0, \ldots, \lfloor \frac{n}{2} \rfloor \right\}\\ 
   
\end{cases}$$
Here $r_\pm = - \frac{\beta}{2} \pm \sqrt{\left( \frac{\beta}{2} \right)^2 + \lambda_{m, k}}, \ \gamma_{m, k} = \sqrt{ \left| \left( \frac{\beta}{2} \right)^2 + \lambda_{m, k}\right|}$ and
$c \in \mathbb{R}, c_1, c_2 \in \mathbb{C}$ are constants.  In the case $n$ is even, there is only the one polygon corresponding to $k= \frac{n}{2}$. 
\end{prop}

We observe from above different behaviour corresponding to the zero eigenvalue $\lambda_{m, 0}.$ In this case, if $\beta=0$ the point $c_1 P_0 + c_2 P_n$ can be moving ($c\ne 0$) or stationary ($c=0$).  Proposition \ref{prop:selfsimilar} shows that each regular polygon $P_k$ is a scaling self-similar solution of (\ref{eqn:polyflow}).  Since the flow \eqref{eqn:polyflow} is invariant under affine transformations, scaling self-similar solutions are in the form of affine transformations of the basis polygons.
For $k>0$ we observe that the self-similar polygon solutions decay with oscillations for small damping coefficient, but for large damping there are no oscillations as they decay.  

\begin{proof}
    Since $\vec{X}(t) = g(t)\vec{X}_0$ is to satisfy equation (\ref{eqn:polyflow}) for all $t$, we use this to find an equivalent expression in terms of $\vec{X}_0.$  We have
    $$\frac{d\vec{X}}{dt} = g'\left( t\right) \vec{X}_0$$
    and
    $$\frac{d^2\vec{X}}{dt^2}= g''\left( t\right) \vec{X}_0$$
    so \eqref{eqn:polyflow} becomes
    $$g''\left( t\right) \vec{X}_0 + \beta g'\left( t\right) \vec{X}_0 = \left( -1\right)^{m+1} g\left( t\right) M^m \vec{X}_0 \mbox{,}$$
    that is,
    $$\left[ \frac{g''\left( t\right)}{g\left( t\right)} + \beta \frac{g'\left( t\right)}{g\left( t\right)}\right] \vec{X}_0 = \left( -1\right)^{m+1} M^m \vec{X}_0 \mbox{.}$$
    The right hand side above is independent of $t$ so the scaling factor $g\left( t\right)$ must satisfy 
  \begin{equation}\label{E:scaling_function_DE}
      g''\left( t\right) + \beta\, g'\left( t\right) - C\, g\left( t\right) = 0
  \end{equation}
    for some constant $C$.  
    
    This being the case, the remaining equation for $\vec{X}_0$ is
    $$C \vec{X}_0 = \left( -1\right)^{m+1} M^m \vec{X}_0 \mbox{.}$$
    For this equation to have a nonzero solution $\vec{X}_0$, it must be that $C$ is an eigenvalue of $\left( -1\right)^{m+1} M^m$.  Thus we have the possibilities $C=\lambda_{m, k}$ for any $k$ with corresponding eigenvectors $P_k$ and $P_{n-k}$. 

    Solving the differential equation  (\ref{E:scaling_function_DE}) with $C=\lambda_{m,k}$ for $k\in \{0,1,\ldots, \lfloor \frac{n}{2} \rfloor\}, \beta$ and the initial condition $g(0)=1$ we obtain the expressions for $g(t)$ as in the statement of the Proposition.
\end{proof}

\begin{prop}\label{prop:rotate_plane}
Consider the family of polygons in the plane, $\vec{X}(t)$ given by 
\begin{equation}\label{eqn:rotateselfsim}
    \vec{X}(t) = \vec{X}_0R(f(t)),
\end{equation}
where $R(f(t))$ represents the rotation of the polygon  $\vec{X}_0$ by angle $f(t),$ with $f(0) = 0.$  
\begin{enumerate}
    \item If $\vec{X}(t)$ satisfies (\ref{eqn:polyflow}) with $\beta = 0$ for all $t,$ then $\vec{X}(t)$ has the form
\begin{equation*}
    \vec{X}(t) = (c_1P_k + c_2P_{n-k})R(\pm \sqrt{-\lambda_{m,k}}t),
\end{equation*}
 for some $k \in \{1,2, \ldots, \lfloor\frac{n}{2}\rfloor\}$ and any constants $c_1, c_2 \in \mathbb{C},$ where $\lambda_{m,k}$ is the corresponding eigenvalue for eigenvectors $P_k$ and $P_{n-k}.$ 
   \item In the case $\beta>0$, there are no nontrivial solutions of \eqref{eqn:polyflow} that evolve by pure rotation.
\end{enumerate}
 \end{prop}

 \begin{remark}
   \begin{enumerate}
   \item We did not include $k=0$ in case 1. above because then $\lambda_{m,0}=0$ and $\vec{X}\left( t\right)=\vec{c}$ for a complex constant $c;$ this is just the trivial solution of \eqref{eqn:polyflow}.
   \item It is not surprising there are no solutions evolving by purely rotation with $\beta>0$ as this corresponds physically to a damping term.
   \end{enumerate}
\end{remark}

\begin{proof}
To establish an expression for the rotation matrix $R(f(t))$ in $SO(2)$ we consider the skew symmetric matrix $S:\mathbb{R} \rightarrow \mathfrak{so}(2)$ such that 
\begin{equation*}
    S(f(t)) = 
    \begin{bmatrix}
        0 & -f(t)\\
        f(t) & 0\\
    \end{bmatrix}.
\end{equation*}
Note that $S(f(t)) = f(t)S(1)$ and so $\frac{d}{dt}S(f(t)) = f'(t)S(1)$ and $\frac{d^2}{dt^2}S(f(t)) = f''(t)S(1).$ Considering the map $\exp:\mathfrak{so}(2) \rightarrow SO(2),$ we have $\exp(S(f(t))) = R(f(t)).$ Furthermore 
\begin{equation*}
    \frac{d}{dt}R(f(t)) = f'(t)S(1)R(f(t)) \text{ and } \frac{d^2}{dt^2}R(f(t)) = \left[f'(t)S(1)\right]^2 R(f(t)) + f''(t)S(1)R(f(t)).
\end{equation*}
We also note that $S(1)^2 = -I_2.$ Therefore for $\vec{X}(t) = \vec{X}_0R(f(t))$ to satisfy (\ref{eqn:polyflow}) we have

\begin{equation}\label{E:rotate_arrange}
    \vec{X}_0\left[ (f''(t) + \beta f'(t))S(1) - (f'(t))^2I_2\right] = (-1)^{m+1}M^m\vec{X}_0.
\end{equation}

The right hand side above is independent of $t$  so for a nonzero solution $\vec{X}$ it must be that $(f''(t) + \beta f'(t))S(1) - (f'(t))^2I_2$ is constant for all $t.$  This requires $f'(t) \equiv b$ for some constant $b$ and so $f''(t) \equiv 0.$
Therefore (\ref{E:rotate_arrange}) becomes
\begin{equation}\label{E:rotate_withbeta}
    \vec{X}_0(\beta bS(1) - b^2I_2) = (-1)^{m+1}M^m\vec{X}_0.
\end{equation}

If $\beta =0,$ then the above reduces to $-b^2\vec{X}_0 = (-1)^{m+1}M^m\vec{X}_0,$ and so
\begin{equation}\label{E:rotate_zerobeta}
\left[(-1)^{m+1}M^m + b^2I_n\right]\vec{X}_0 = 0_{n\times 2}.
\end{equation}

For a nontrivial solution $\vec{X}_0$ we require 
\begin{equation*}
    \det\left[(-1)^{m+1}M^m + b^2I_n\right] = \prod_{k=0}^{\lfloor\frac{n}{2}\rfloor}(\lambda_{m,k} +b^2) = 0.
\end{equation*}
Hence $b= \pm \sqrt{-\lambda_{m,k}}$ for some $k\in \{0,1,\ldots, \lfloor\frac{n}{2}\rfloor\}.$  Given that $-b^2 = \lambda_{m,k},$ the null space of $(-1)^{m+1}M^m + b^2I_n$ is spanned by $P_k$ and $P_{n-k}.$ Therefore $\vec{X}_0 = c_1P_k + c_2P_{n-k}$ for complex constants $c_1, c_2$ and the rate of rotation $f(t)$ is given by $f(t) = \pm \sqrt{-\lambda_{m,k}}t$ in this case. Note that for $k=0$ we have $\lambda_{m,0} = 0$ and $P_0 = (1,\ldots, 1)^T$ and so $f(t)\equiv 0$ and $\vec{X}_0$ is a constant vector. 

We now consider the case $\beta > 0.$ From (\ref{E:rotate_withbeta}) we have
\begin{equation} \label{E:ssr}
    \vec{X}_0[\beta bS(1) - b^2I_2]^2  = (b^4 - \beta b^2)\vec{X}_0 - 2\beta b^3 \vec{X}_0S(1) = M^{2m}\vec{X}_0.
\end{equation}
Letting $\vec{X}_0 = [\vec{x}\  \vec{y}]$ we therefore have
\begin{equation*}
 (b^4 - \beta^2 b^2)[\vec{x}\  \vec{y}] - 2\beta b^3[\vec{y}\ -\vec{x}] = M^{2m}[\vec{x}\  \vec{y}]
\end{equation*}
which, with a rearrangement is equivalent to

\begin{equation*}
\left[(M^{2m} - (b^4 - \beta^2b^2)I_n)^2 + 4\beta^2b^6\right]\vec{X}_0 = 0_{n\times 2}.
\end{equation*}
For a non trivial $\vec{X}_0$ we require
\begin{equation*}
    \det\left\{ \left[ M^{2m} - (b^4 - \beta^2b^2)I_n\right]^2 + 4\beta^2b^6\right\} = \prod_{k=0}^{\lfloor\frac{n}{2}\rfloor}\left\{ \left[\lambda_{m,k}^2 - (b^4 - \beta^2b^2)\right]^2 + 4\beta^2b^6\right\} = 0.
\end{equation*}
Each factor in the product above is the sum of two squares.  Since $\beta>0$, the only chance of a zero factor is if $b=0$ which implies $f\left( t\right) \equiv 0$.  Substituting $b=0$ back into \eqref{E:ssr} we find only the trivial solution $\vec{X}(t) = \vec{c}$ for any complex constant $c.$
\end{proof}

\begin{prop}\label{prop:selfsim_plane_translate}
  Consider the family of polygons in the plane, $\vec{X}(t)$ given by 
  \begin{equation}\label{eqn:translateselfsim2}
    \vec{X}(t) = \vec{X}_0 + \vec{h}(t),
\end{equation}
where $\vec{h}(t)$ represents the translation of the polygon $\vec{X}_0$ and $\vec{h}(0) = \vec{0}.$  If $\vec{X}(t)$ satisfies (\ref{eqn:polyflow}) for all $t$ then $\vec{X}_0$ corresponds to a single point in the plane. That is, there are no non-trivial self-similar solutions by translation under the semi-discrete polyharmonic flow.  
\end{prop}

\begin{proof}
Since $\vec{h}(t)$ translates each vertex of $\vec{X}_0$ in the same way, it follows that $\vec{h}(t)$ must be a vector in $\mathbb{C}^n$ with every entry equal to the same function of $t.$ Since (\ref{eqn:translateselfsim2}) is to satisfy (\ref{eqn:polyflow}), it follows that $\vec{h}\left( t\right)$ must satisfy
\begin{equation*}
    \vec{h}''(t) + \beta \vec{h}'\left( t\right) = (-1)^{m+1}M^m \left[ \vec{X}_0 + \vec{h}(t) \right]
    = (-1)^{m+1}M^m\vec{X}_0.
\end{equation*}
Again, the right hand side is independent of $t$, so each element $h(t)$ of $\vec{h}\left( t\right)$ must satisfy
$$h''\left( t\right) + \beta\, h'\left( t\right) = C$$
for some complex constant $C$.  In view of Lemmas \ref{lem:constant_vec} and \ref{lem:nullspace_matrix}, the only solutions of 
$$(-1)^{m+1}M^m \vec{X}_0 = \vec{C}$$
are the constants $\vec{X}_0 = \vec{a}_0$, corresponding to trivial solutions of \eqref{eqn:polyflow}.
 \end{proof}

\noindent \emph{Remark:}  
One can solve the ODE for $h$ to obtain the expression for the path of translating point in both the $\beta=0$ and $\beta>0$ cases.  The results are, for $\beta=0$,
$$h\left( t\right) = d\, t$$
and for $\beta>0$,
$$h\left( t\right) = \frac{d}{\beta}\left( 1 - e^{-\beta \, t}\right)$$
for constant $d$ not equal to zero.

 \subsection{Planar solutions for general initial data}\label{sec:flow_solutions}

In this subsection we move from considering those solutions to \eqref{eqn:polyflow} that move self-similarly to general solutions given specified initial data.  As the governing equation is hyperbolic there are natural analogues of initial position and velocity; the former is clear in our setting while the latter gives rise to several options.  Consequently we specify our results in this section for given initial polygon $\vec{X}_0$ only and include free parameters that can be determined from an appropriate additional condition.  Two possible such conditions are outlined in the remark below.

\begin{theorem} \label{thm:planenodamp}
Given an initial polygon $\vec{X}_0 = \sum_{k=0}^{n-1} \alpha_k^0 P_k$ with $n$ vertices in $\mathbb{R}^2$ and any $m\in \mathbb{N}$, the equation \eqref{eqn:polyflow} with $\beta=0$ and initial data $\vec{X}\left(0 \right) = \vec{X_0}$ has a unique solution given by
\begin{equation} \label{E:flowsolnnodamp}
  \vec{X}(t) = \left( \alpha_0^0 + a_0\, t\right) P_0 + \sum_{k=1}^{n-1} \left[ \alpha_k^0 \cos\left( \sqrt{- \lambda_{m, k}} t\right) + a_k \sin \left( \sqrt{- \lambda_{m, k}} t\right) \right] P_k,
\end{equation}
where the $a_k$, $k=0, \ldots, n-1$ are arbitrary constants.
\end{theorem}

\begin{proof}
The proof is similar to the parabolic flow cases in  \cite{chow2007semidiscrete} for $m=1$ and in \cite{MM24} for general $m$.  The set of eigenvectors $\{P_k\}_{k=0}^{n-1}$ of $(-1)^{m+1}M^m$ forms a basis for $\mathbb{C}^n$. Therefore by considering $\vec{X}(t) \in \mathbb{C}^n$ we can write our polygon in the form 
\begin{equation} \label{E:Xform}
  \vec{X}(t) = \sum_{k=0}^{n-1}\alpha_k(t)P_k
  \end{equation}
  where the coefficients $\alpha_k(t)$ are complex.  We have 
$$\frac{d\vec{X}}{dt}  = \sum_{k=0}^{n-1} \alpha_k'\left( t\right) P_k$$
and
$$\frac{d^2\vec{X}}{dt^2}  = \sum_{k=0}^{n-1}\alpha_k''\left( t\right)  P_k \mbox{.}$$
Since $\vec{X}(t)$ satisfies (\ref{eqn:polyflow}) with $\beta=0$, this gives\begin{equation*}
    \sum_{k=0}^{n-1} \alpha_k''\left( t\right) P_k  
     = (-1)^{m+1}M^m\sum_{k=0}^{n-1}\alpha_k(t)P_k
     = \sum_{k=0}^{n-1}\alpha_k(t)\lambda_{m,k}P_k.
\end{equation*}
Hence
$$\alpha_k''(t) = \lambda_{m, k} \alpha_k(t)$$
for each $t$ which implies
$$\alpha_0(t) = c_0 + a_0\, t$$
while for the other $k$,
$$\alpha_k(t) = c_k\, \cos \left( \sqrt{- \lambda_{m, k}} t\right) + a_k \sin \left( \sqrt{- \lambda_{m, k}} t\right) \mbox{,}$$
for constants $c_k, a_k, k=0, 1, \ldots, n-1$.  The result follows in view of the initial coefficients.
\end{proof}

\begin{remark}
    The appearance of arbitrary constants $a_0, \ldots a_{n-1}$ in the solution formula \eqref{E:flowsolnnodamp} is not surprising.  Prescribing only the initial polygon does not give provide enough information to solve \eqref{eqn:polyflow} uniquely.  There are of course many ways extra information can be given yielding a unique solution.  Two of these are 
    \begin{enumerate}
        \item A `zero initial velocity' condition.  Clearly, in view of \eqref{E:flowsolnnodamp}, this will result in $a_0=a_1=\ldots = a_{n-1}=0$.
        \item A `prescribed polygon at later time' condition.  In other words, not only do we specify the initial polygon, but we also specify another polygon at some later time.  In view of the arguments of the sine functions in \eqref{E:flowsolnnodamp}, this can be messy in general, however, suppose for a specific example we required
        $$\vec{X}\left( \frac{\pi}{2\sqrt{-\lambda_{m, 1}}} \right) = P_0 + P_1 \mbox{.}$$
        From \eqref{E:flowsolnnodamp}, we must therefore have
        $$\alpha_0^0 + a_0 \frac{\pi}{2\sqrt{-\lambda_{m, 1}}} = 1 \mbox{, } 
          a_1 = 1 \mbox{ and } \alpha^0_k \cos \left( \frac{\pi}{2}\sqrt{ \frac{\lambda_{m, k}}{\lambda_{m, 1}}}\right) +  a_k\sin \left( \frac{\pi}{2}\sqrt{ \frac{\lambda_{m, k}}{\lambda_{m, 1}}}\right) = 0,\text{ for } k= 2, \ldots, n-1 \mbox{.}$$
        This kind of specification of $\vec{X}$ at another time is also a kind of approach to the discrete Yau problem where we consider the solution only on a finite time interval.

    \end{enumerate}
    \end{remark}

    Next we consider the case of $\beta>0$.  The damping ensures convergence to a point that, under certain conditions and under appropriate rescaling, is an affine transformation of a regular polygon.  This is fundamentally different behaviour from the undamped case where the solution \eqref{E:flowsolnnodamp} exhibits continued undamped oscillations.

\begin{theorem}\label{thm:planedamp}
Given an initial polygon $\vec{X}_0 = \sum_{k=0}^{n-1} \alpha_k^0 P_k$ with $n$ vertices in $\mathbb{R}^2$ and any $m\in \mathbb{N}$, the equation \eqref{eqn:polyflow} with $\beta >0$ constant and initial data $\vec{X}\left(0 \right) = \vec{X}_0$ has a unique solution given by
\begin{equation} \label{E:flowsolndamp}
  \vec{X}(t) = \left[ \alpha_0^0 + \frac{a_0}{\beta} - \frac{a_0}{\beta} e^{-\beta\, t} \right] P_0 + \sum_{k=1}^{n-1}\alpha_k(t) P_k,
\end{equation}

where
\begin{equation} \label{E:alphak}
  \alpha_k\left( t\right) = \begin{cases}
        \alpha^0_k e^{r_{+m, k} t} + a_k \left(e^{r_{-m, k} t} - e^{r_{+m, k} t}\right) &\mbox{ for } \left| \lambda_{m, k} \right| < \frac{\beta^2}{4}\\
        (\alpha^0_k + a_k t) e^{r_{m, k} t} &\mbox{ if } \lambda_{m, k} = - \frac{\beta^2}{4}\\
        e^{-\frac{\beta}{2} t} \left[ \alpha^0_k \cos \left( \gamma_{m, k} t \right) + a_k \sin\left( \gamma_{m, k} t\right) \right],     &\mbox{ for } \left| \lambda_{m, k} \right| > \frac{\beta^2}{4}    
  \end{cases}
  \end{equation}
 $$r_{\pm m, k} = - \frac{\beta}{2} \pm \sqrt{\left( \frac{\beta}{2}\right)^2 + \lambda_{m, k}} $$ and $$\gamma_{m, k} = \sqrt{-\lambda_{m, k} - \left( \frac{\beta}{2}\right)^2} \mbox{.}$$
  The constants $a_k$ for $k = 0,\ldots, n-1$ are arbitrary. The solution exists for all time and converges exponentially to a point.  

When the dominant eigenvalue $\lambda_{m,d}$ for some $d\in \{1,\ldots, \lfloor \frac{n}{2}\rfloor\}$ in the expression of $\vec{X}_0$ satisfies the condition $|\lambda_{m,d}| \leq \frac{\beta^2}{4},$ then under appropriate rescaling, the solution is asymptotic as $t\rightarrow \infty$ to an affine transformation of a regular polygon with $n$ vertices. Otherwise if $|\lambda_{m,d}| > \frac{\beta^2}{4},$ then the solution in (\ref{E:flowsolndamp}) exhibits continued oscillating behaviour as it shrinks to a point.
\end{theorem}

\noindent \emph{Remark:} In view of \eqref{E:alphak}, all modes of the solution are exponentially decaying except for $P_0$.  

\begin{proof}
Again writing  $\vec{X}(t) = \sum_{k=0}^{n-1}\alpha_k(t)P_k$, since $\vec{X}(t)$ satisfies (\ref{eqn:polyflow}) with $\beta>0$, this gives
\begin{equation*}
    \sum_{k=0}^{n-1}\left[ \alpha_k''\left( t\right) + \beta \alpha_k'\left( t\right) \right] P_k  
     = (-1)^{m+1}M^m\sum_{k=0}^{n-1}\alpha_k(t)P_k
     = \sum_{k=0}^{n-1}\alpha_k(t)\lambda_{m,k}P_k.
\end{equation*}

Therefore for each $k$,
$$\alpha_k''\left( t\right) + \beta \, \alpha_k'\left( t\right)  = \lambda_{m,k}\, \alpha_k(t).$$ 

Thanks Solving the above equation for different cases of the eigenvalues $\lambda_{m,k}$ in relation to $\beta,$ solves the coefficients $\alpha_k(t)$ as given in (\ref{E:alphak}).

Since $\beta > 0$ and each $\lambda_{m,k} <0$ for non zero $k,$ then each $\alpha_k(t)$ goes to zero as $t \to \infty$ for each $k = 1,\ldots, n-1.$  Therefore 
\begin{equation*}
    \lim_{t\to \infty}\vec{X}(t) = \left[\alpha_0^0+ \frac{a_0}{\beta}\right]P_0,
\end{equation*}

that is, each vertex of the polygon converges to the same point, by $\alpha_0^0+ \frac{a_0}{\beta}$ where $\alpha_0^0$ is the complex constant given by the initial polygon, and $a_0$ is an arbitrary complex constant.

To determine the limiting shape of the polygon, we consider an appropriately scaled and translated version of $\vec{X}(t)$ given by 
\begin{equation*}
    \vec{Y}(t) = g(t)\left(\vec{X}(t) - \left[\alpha_0^0+ \frac{a_0}{\beta} - \frac{a_0}{\beta}e^{-\beta t}\right]P_0\right).
\end{equation*}

The scaling factor $g(t)$ is determined by the value of the damping term $\beta$ and the relationship of the eigenvalues $\lambda_{m,k},$ to this damping term, such that this therefore determines the coefficient terms as given in (\ref{E:alphak}).

Suppose that for the dominant eigenvalue $\lambda_{m,1},$ we have $|\lambda_{m,1}| < \frac{\beta^2}{4}.$ Then we choose the scaling factor to be 
\begin{equation*}
    g(t) = e^{-r_{+m,1}t} = e^{\left(\frac{\beta}{2} -\sqrt{\frac{\beta^2}{4} + \lambda_{m,1}}\right)t}.
 \end{equation*}

Note that for $k = 2,\ldots, n-1$ we have
$$r_{\pm m,k} - r_{+m,1} = \pm \sqrt{\frac{\beta^2}{4}+ \lambda_{m,k}} - \sqrt{\frac{\beta^2}{4}+ \lambda_{m,1}} < 0$$ and 
$$-\frac{\beta}{2} - r_{+m,1} = - \sqrt{\frac{\beta^2}{4} + \lambda_{m,1}} < 0.$$

Therefore for $k =2\ldots, n-1$ the expression $e^{-r_{+m,1}t}\alpha_k(t)$ will go to zero as $t\to \infty$ and we have 
\begin{equation*}
    \lim_{t\to \infty}\vec{Y}(t) = (\alpha_1^0 - a_1)P_1 + (\alpha_{n-1}^0 - a_{n-1})P_{n-1}.
\end{equation*}
Therefore $\vec{Y}(t)$ converges to an affine transformation of $P_1.$

If we have the condition $\lambda_{m,1} = \frac{\beta^2}{4},$
then we choose the scaling factor $g(t) = \frac{e^{\frac{\beta}{2}t}}{t}.$
We therefore obtain

\begin{equation*}
    \lim_{t\to \infty}\vec{Y}(t) = a_1P_1 + a_{n-1}P_{n-1},
\end{equation*}
where the limiting shape is again an affine transformation of $P_1.$

In the case of the original polygon being orthogonal to $P_1,$ then we instead consider the next dominant eigenvalue $\lambda_{m,d}$ for some $d\in \{2,\ldots \lfloor \frac{n}{2} \rfloor\}$ where the initial polygon is not orthogonal to $P_d.$ The scaling factor $g(t)$ is therefore chosen in the same way as described above, but instead involving $\lambda_{m,d}.$

In the case of the dominant eigenvalue with property $|\lambda_{m,d}|>\frac{\beta^2}{4}$ then taking the scaling factor to be $g(t) = e^{\frac{\beta }{2}t}$ gives
$$\vec{Y}(t) = \sum_{k=1}^{n-1}\left[ \alpha^0_k \cos \left( \gamma_{m, k} t \right) + a_k \sin\left( \gamma_{m, k} t\right) \right].$$ Therefore in this case we have oscillating behaviour of the polygon as it shrinks to a point. 
\end{proof}

\begin{remark}
As we did earlier for the undamped case, we can supplement the initial condition in Theorem \ref{thm:planenodamp} with an additional condition to ensure a unique solution.  In the case of prescribing a second polygon at a later time, we create a hyperbolic flow that begins with one polygon, transitions through a second polygon at some fixed time and then converges exponentially to a point that, under rescaling, is an affine transformation of a regular polygon for certain values of $\beta$.
\end{remark}

Figure \ref{fig:general_n=5} depicts the evolution of a pentagon under the semi-discrete hyperbolic polyharmonic flow at a series of time steps for select values of $m.$ A damping term of $\beta=4$ is included. The figures demonstrate the behaviour of the flow as the polygon shrinks and converges to an affine transformation of the regular polygon. In the case of $m=1$ we have $|\lambda_{m,k}|<\frac{\beta^I 2}{4}$ for all eigenvalues. In the case of $m=2$ and $m=3,$ only the dominant eigenvalue $\lambda_{m,1}$ satisfies this condition, and so the terms in our solution that involve the eigenvectors associated with non-dominant eigenvalues, include coefficients with oscillating expressions as set out in (\ref{E:alphak}). Convergence is faster for higher $m.$

Figure \ref{fig:general_n=5_intermediate} depicts the evolution of the same pentagon as given in Figure \ref{fig:general_n=5} under the semi-discrete hyperbolic polyharmonic flow at the same overlay-ed time steps and for select values of $m.$ In this case however, non zero arbitrary coefficients $a_k$ are prescribed such that the polygon flows to a specific polygon at a particular time value, before shrinking to a point. In the examples given in Figure \ref{fig:general_n=5_intermediate}, the prescribed intermediate polygon is depicted in red.

\begin{figure*}

  \begin{subfigure}[b]{0.31\textwidth}
    \includegraphics[width=\textwidth]{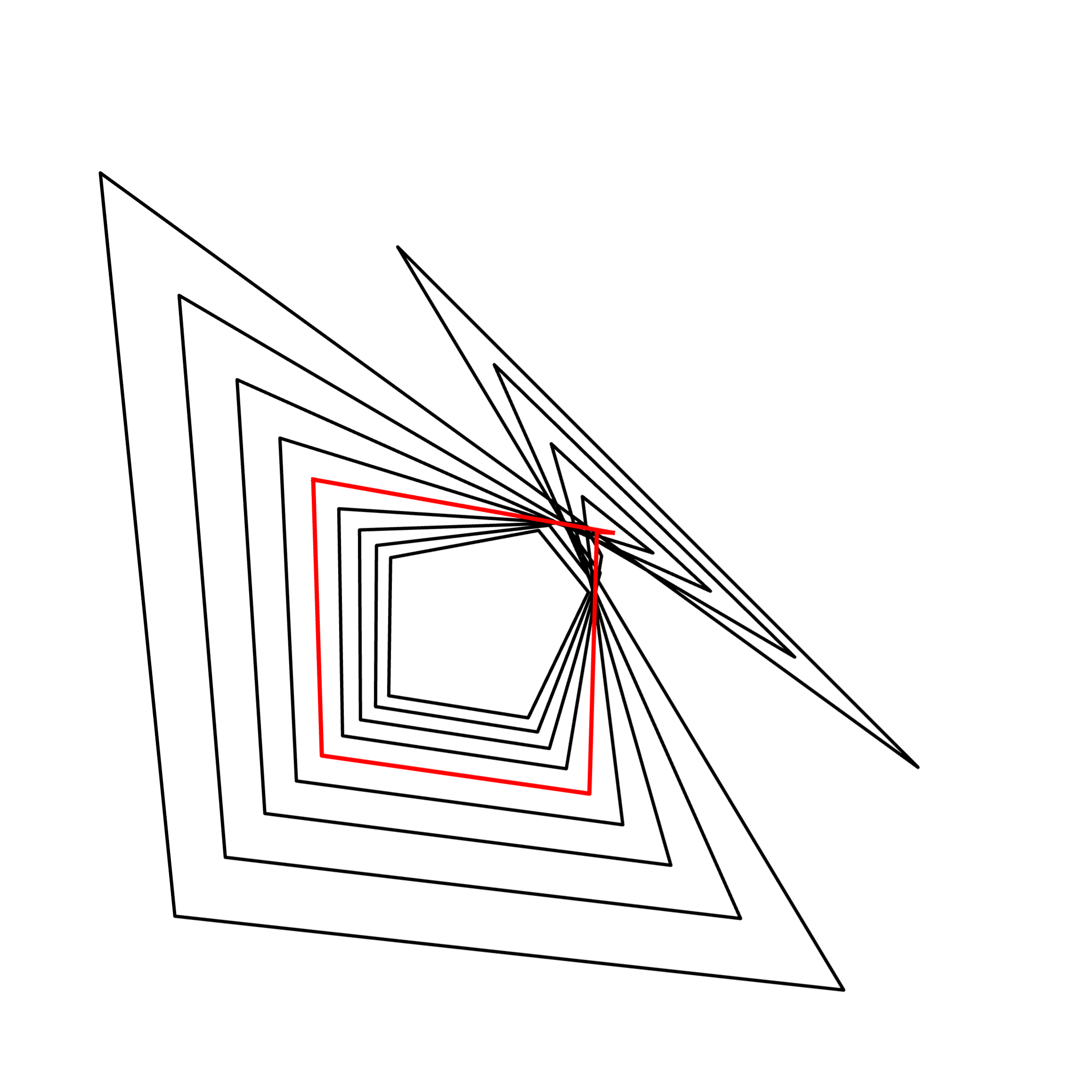}
    \caption{$m=1, \beta = 4$}
    \label{fig:general_n=5_m=1}
  \end{subfigure}
  \hfill
  \begin{subfigure}[b]{0.31\textwidth}
    \includegraphics[width=\textwidth]{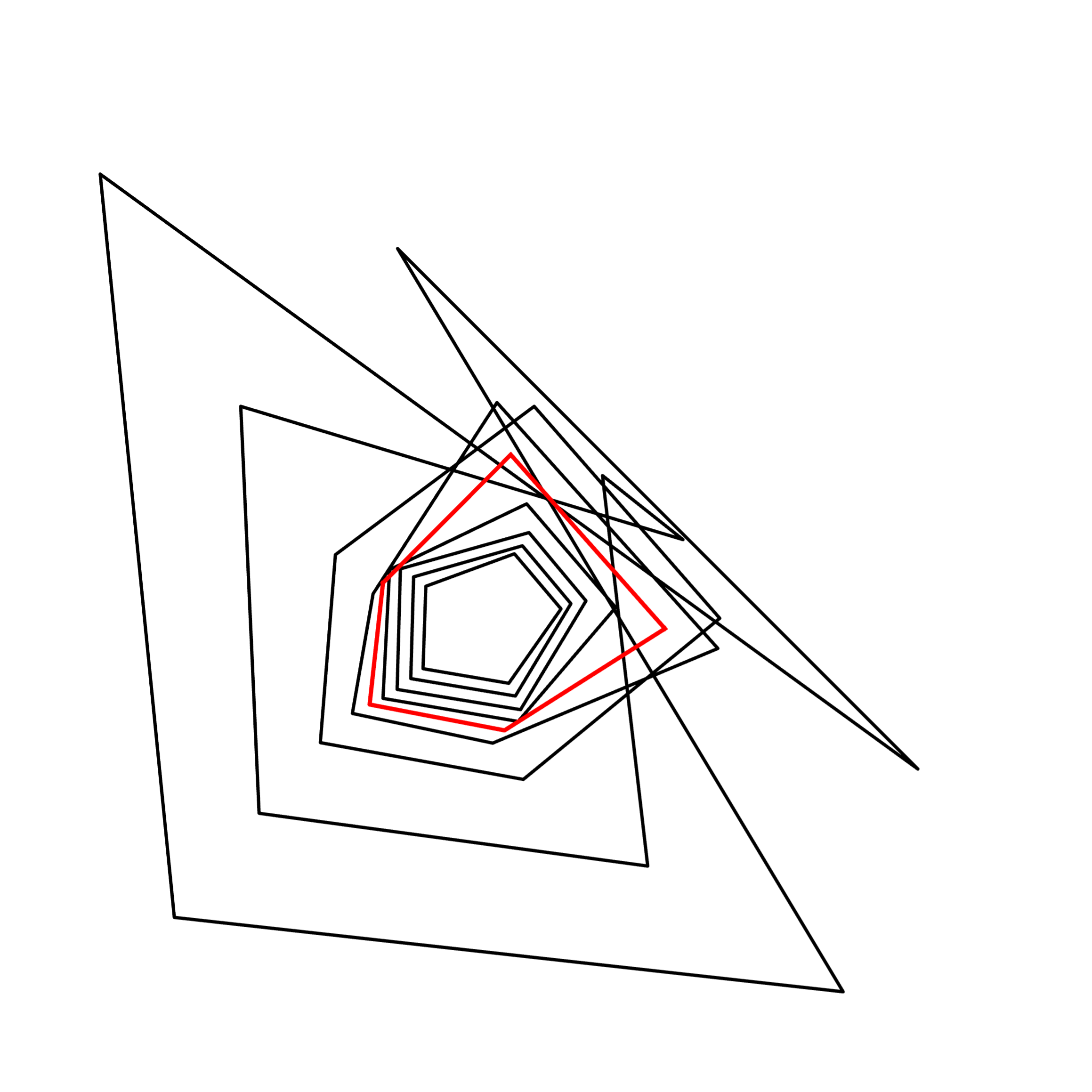}
    \caption{$m=2, \beta = 4$}
    \label{fig:general_n=5_m=2}
  \end{subfigure}
  \hfill
   \begin{subfigure}[b]{0.31\textwidth}
    \includegraphics[width=\textwidth]{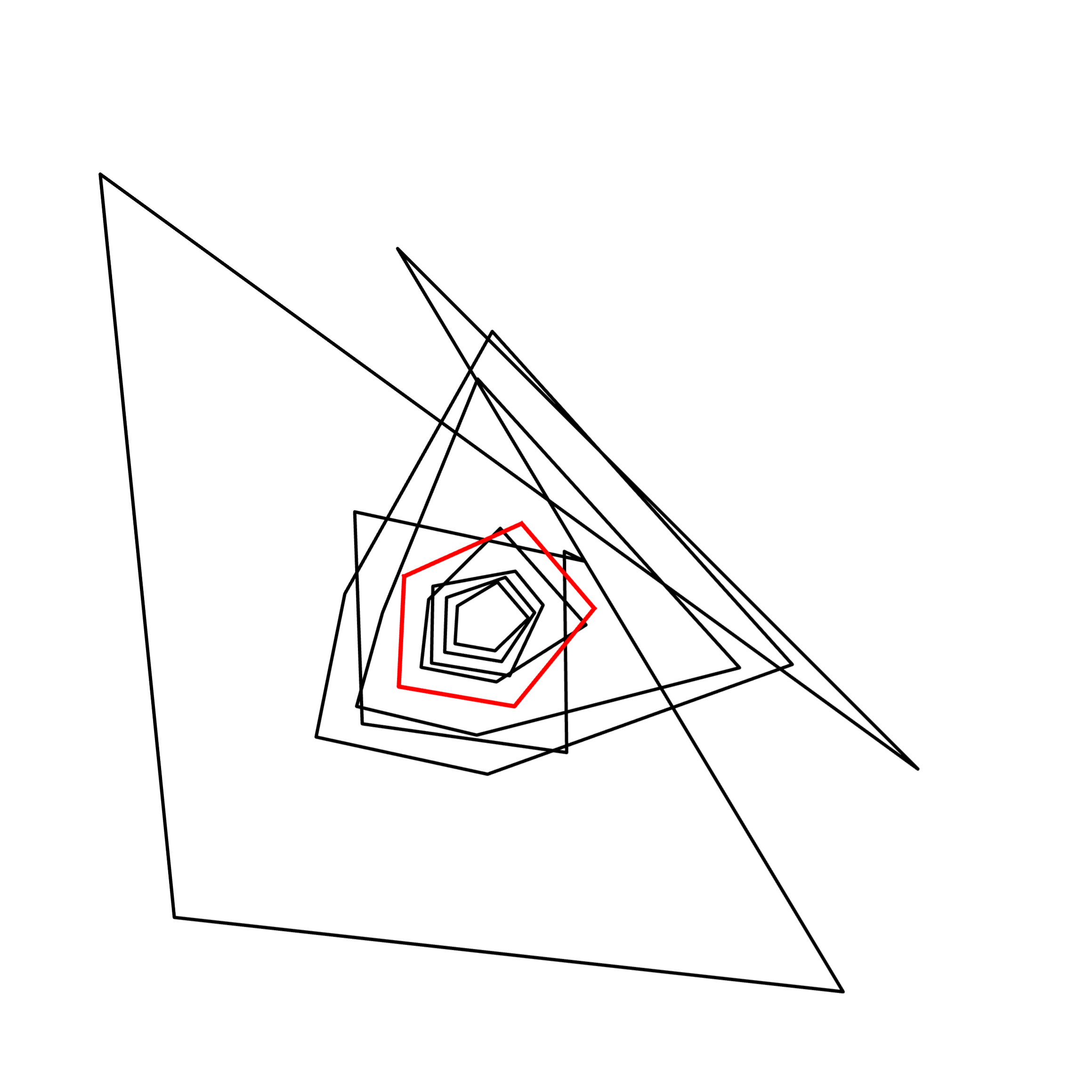}
    \caption{$m=3, \beta = 4$}
    \label{fig:general_n=5_m=3}
  \end{subfigure}
 
  \caption{Evolution of a pentagon under the semi-discrete hyperbolic polyharmonic flow for different values of $m$ and where we have a damping term $\beta = 4.$ All arbitrary constants $a_k$ are chosen to be zero. Distinct time steps of the evolution are shown superimposed over the initial polygon. The same time step values are used for each case of $m$. The polygon at time step $t=1.2$ is shown in red in each case which is to provide comparison to Figure \ref{fig:general_n=5_intermediate} where constants $a_k$ are chosen so that the polygon first flows to a intermediate polygon at time $t=1.2$ before shrinking to a point.} \label{fig:general_n=5}
\end{figure*}

\begin{figure*}

  \begin{subfigure}[b]{0.31\textwidth}
    \includegraphics[width=\textwidth]{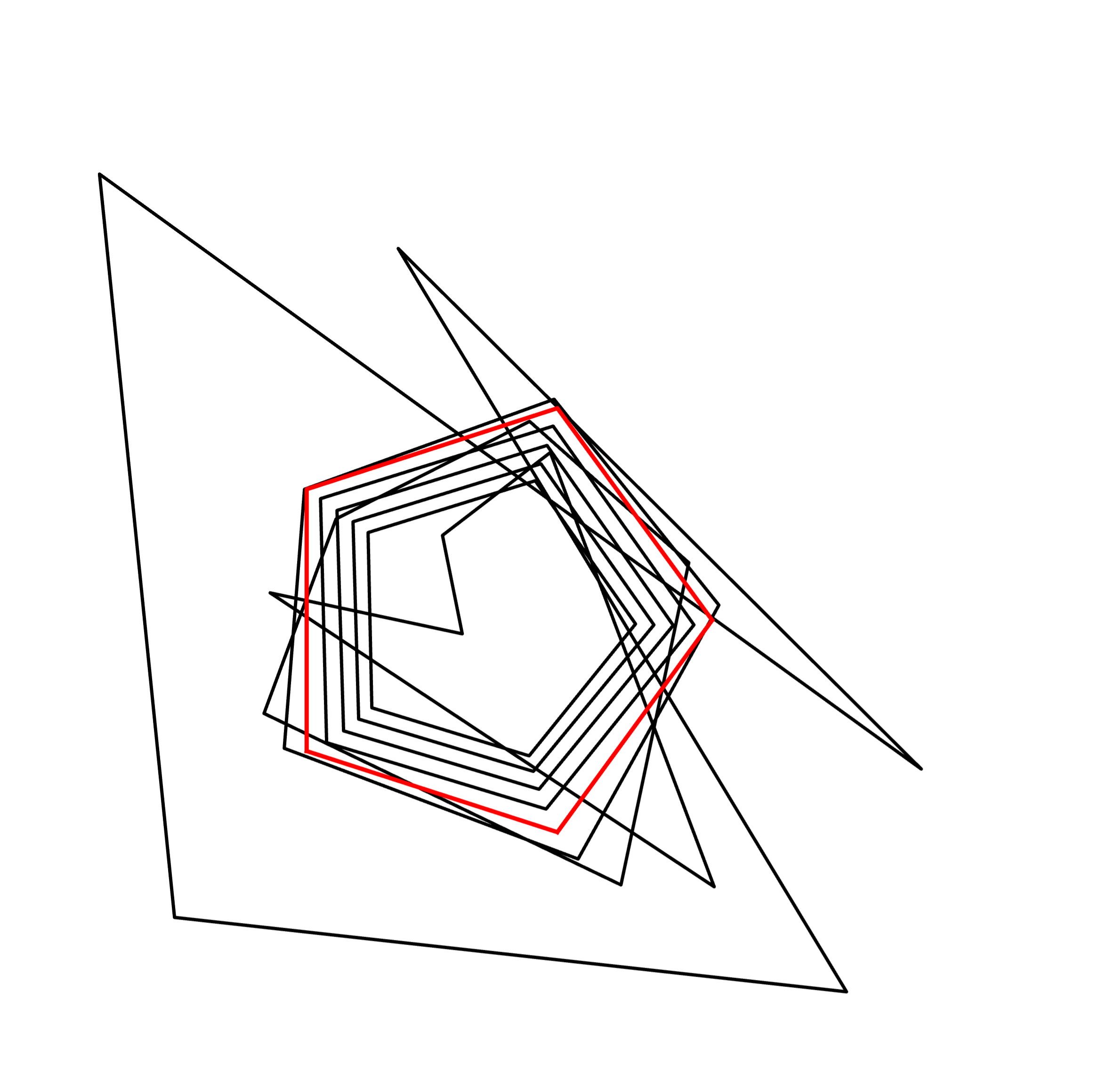}
    \caption{$m=1, \beta = 4$}
    \label{fig:general_n=5_m=1_inter}
  \end{subfigure}
  \hfill
  \begin{subfigure}[b]{0.31\textwidth}
    \includegraphics[width=\textwidth]{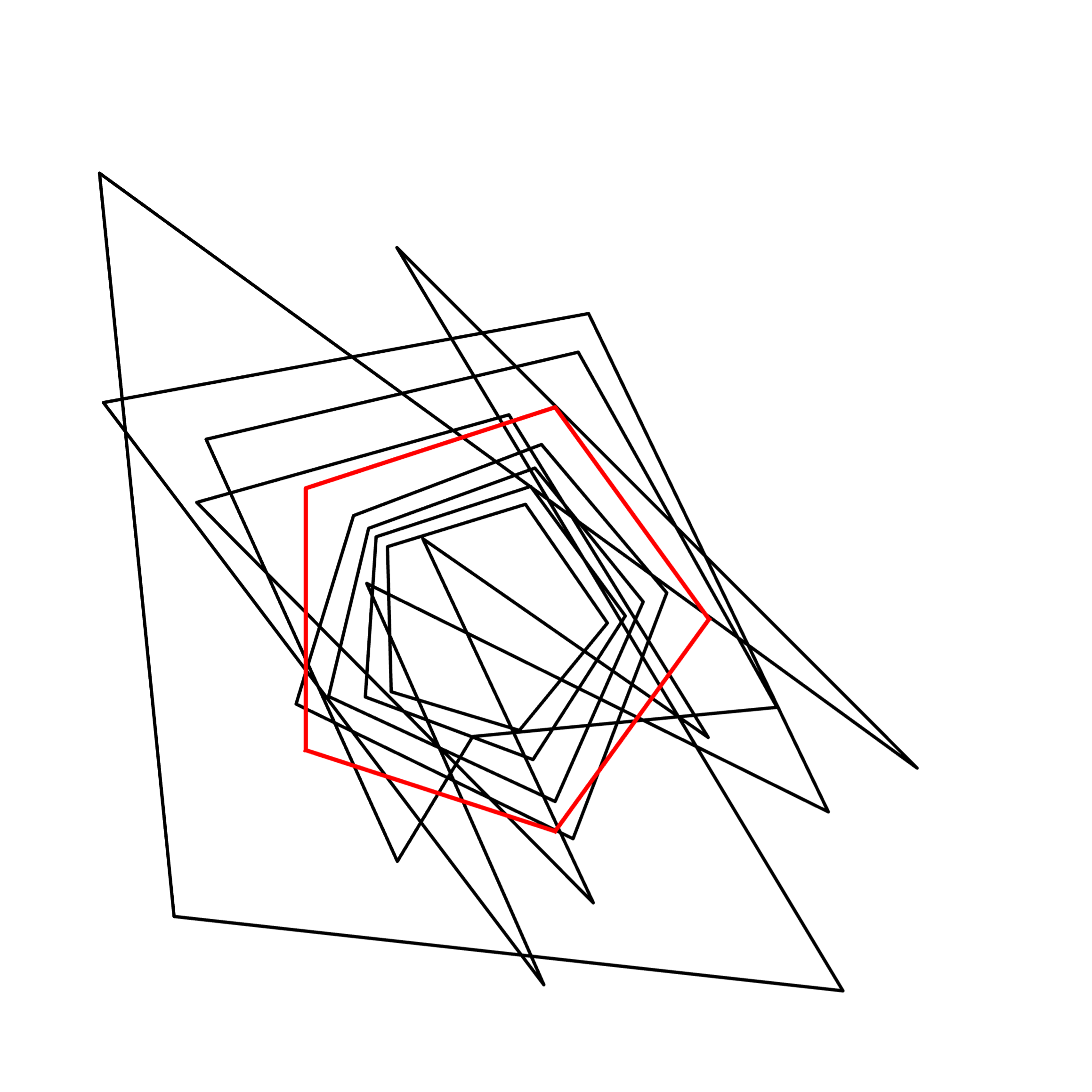}
    \caption{$m=2, \beta = 4$}
    \label{fig:general_n=5_m=2_inter}
  \end{subfigure}
  \hfill
   \begin{subfigure}[b]{0.31\textwidth}
    \includegraphics[width=\textwidth]{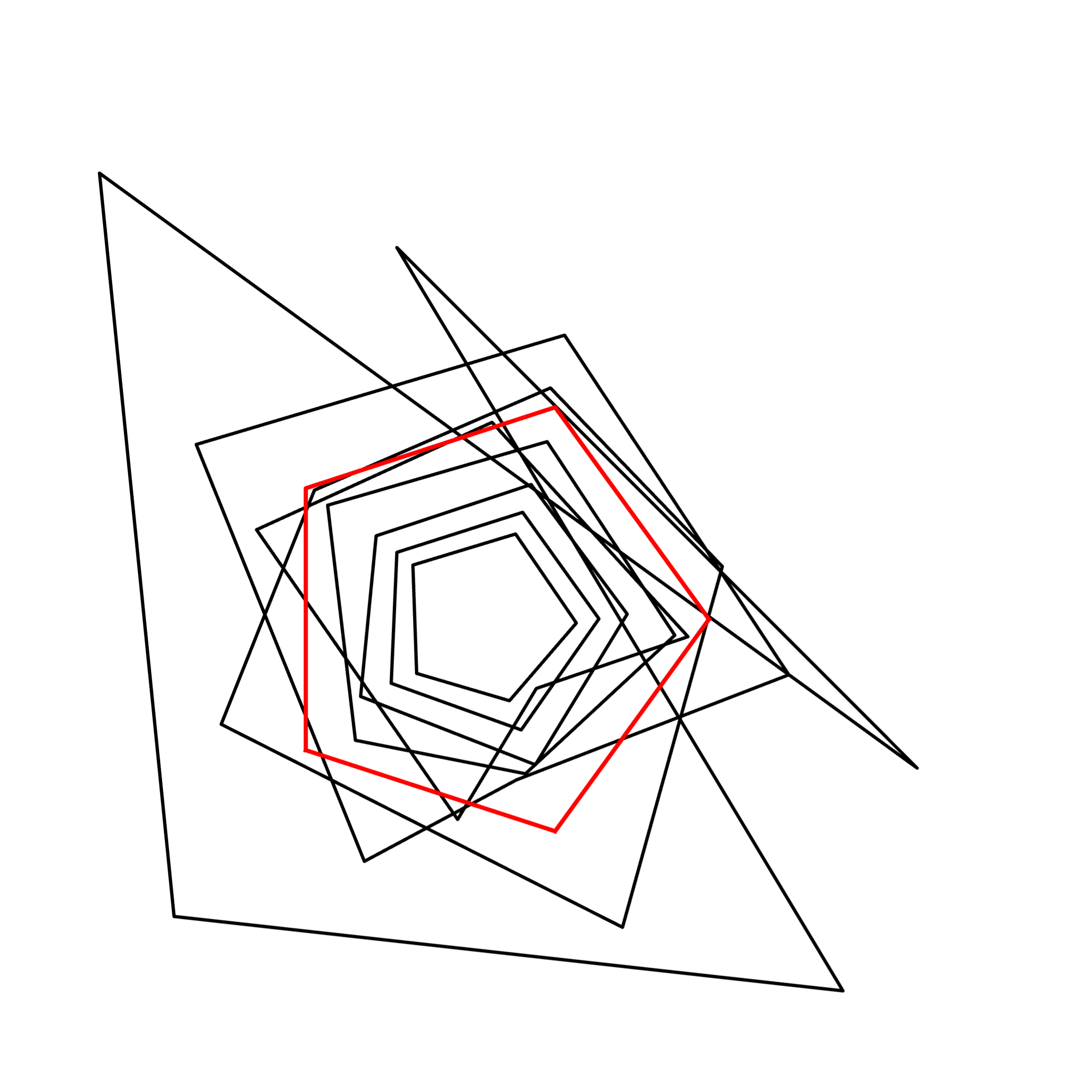}
    \caption{$m=3, \beta = 4$}
    \label{fig:general_n=5_m=3_inter}
  \end{subfigure}
 
  \caption{Evolution of a pentagon under the semi-discrete hyperbolic polyharmonic flow for different values of $m$ and where we have a damping term $\beta = 4.$ The constants $a_k$ are chosen so that at time $t=1.2$ we have $\vec{X}(1.2) = \alpha_0^0P_0+ 3P_1$. Distinct time steps of the evolution are shown superimposed over the initial polygon with the intermediate prescribed polygon at time $t=1.2$ shown in red. The same time step values are used for each case of $m$.} \label{fig:general_n=5_intermediate}
\end{figure*}

\section{Solutions in higher codimension}\label{sec:higher_codimension}

To consider the flow in higher codimension, we set up similarly as in \cite{chow2007semidiscrete}.  Let each vertex $X_j \in \mathbb{R}^p$ be denoted as $X_j = (x_{1j}, x_{2j},\ldots, x_{pj}).$ Consider the $i$th coordinate for every vertex in the polygon, for $i = 1, 2, \ldots, p$, we can define
$$\vec{x}_i = (x_{i0}, x_{i1}, \ldots, x_{i(n-1)})^T,$$ which is a vector in $\mathbb{R}^n.$
Therefore
\begin{equation}\label{eqn:polygon_coordinates}
    \vec{X} = (\vec{x}_1\  \cdots\  \vec{x}_p).
\end{equation}

For $k = 0, 1, \ldots, \lfloor \frac{n}{2}\rfloor,$ let us define the following vectors in $\mathbb{R}^n$
\begin{equation}\label{eqn:ck}
\vec{c}_k  = \left(1, \cos\left(\frac{2\pi k}{n}\right), \cos\left(\frac{4\pi k}{n}\right), \ldots, \cos\left(\frac{2(n-1)\pi k}{n} \right)\right)^T 
\end{equation}
\begin{equation}\label{eqn:sk}
\vec{s}_k  = \left(0, \sin\left(\frac{2\pi k}{n}\right), \sin\left(\frac{2\pi k}{n}\right), \ldots, \sin\left(\frac{2(n-1)\pi k}{n} \right)\right)^T \mbox{.}
\end{equation}
The vectors $\vec{c}_k$ and $\vec{s}_k$ are the real and imaginary parts respectively of eigenvector $P_k.$ Thinking of each entry of $P_k$ as expressed as a vector in $\mathbb{R}^2,$ we have $P_k = (\vec{c}_k\  \vec{s}_k).$ Furthermore, nonzero elements from the set $\{\vec{c}_k, \vec{s}_k\}_{k=0, 1, \ldots, \lfloor \frac{n}{2} \rfloor}$ are mutually orthogonal and form a basis for $\mathbb{R}^n.$ 

Therefore each $\vec{x}_i$ for $i=1,\ldots, p$ can be expressed as 
\begin{equation*}
    \vec{x}_i = \sum_{k=0}^{\lfloor \frac{n}{2} \rfloor}(\alpha_{ik}\vec{c}_k + \beta_{ik}\vec{s}_k)
\end{equation*}
for real coefficients $\alpha_{ik}$ and $\beta_{ik}.$

A family of polygons can therefore be expressed as 

 \begin{equation}\label{eqn:highcodim_polygons}
    \vec{X}(t)  = (\vec{x}_1(t)\  \cdots\  \vec{x}_p(t))
             = \sum_{k=0}^{\lfloor \frac{n}{2} \rfloor}(\vec{c}_k\ \vec{s}_k)
            \begin{bmatrix} 
            \alpha_{1k}(t) & \cdots & \alpha_{pk}(t)\\
            \beta_{1k}(t) & \cdots & \beta_{pk}(t)\\
            \end{bmatrix}.
\end{equation}

\subsection{Self-similar solutions in higher codimension}

We get obtain the behaviour of self-similar solutions in the higher codimension that we see in the plane polygon case. In regards to solutions that are self-similar by scaling, these solution polygons are planar in $\mathbb R^p.$ There are no non-trivial self similar solutions by translation. In the case of self similar solutions by rotation, when a damping term $\beta>0$ is involved again we have no non-trivial solutions. In the rotation case where $\beta=0,$ pure rotations are possible.  We provide a simple construction of a planar rotator, but do not give a complete classification of general rotators.

\begin{prop}
   If a family $\vec{X}(t)$ of polygons with $n$ vertices in $\mathbb{R}^p$ is a self-similar scaling solution to the flow (\ref{eqn:polyflow}), then $\vec{X}(t)$ has the form 
   \begin{equation}\label{eqn:selfsim_highcodim_scale}
       \vec{X}(t) = g(t)(\vec{c}_k\ \vec{s}_k)
    \begin{bmatrix} 
    \alpha_{1} & \cdots & \alpha_{p}\\
    \beta_{1} & \cdots & \beta_{p}\\
    \end{bmatrix}\\,
   \end{equation}
   for any $k \in \{1, 2, ..., \lfloor \frac{n}{2} \rfloor\}$ and real constants $\alpha_j$ and $\beta_j$ for $j = 1,\ldots, p,$ where 
   
   \begin{equation}\label{eqn:selfsim_scale_cases}
       g\left( t\right) = \begin{cases}
    c\, t + 1 & \mbox{ for }\beta=0 \mbox{ and } k=0\\
    \frac{c}{\beta} + \left[ 1 - \frac{c}{\beta} \right] e^{-\beta\, t} & \mbox{ for } \beta > 0 \mbox{ and } k =0\\
 e^{-\frac{\beta}{2}t} \left[ \cos \left( \gamma_{m, k}\, t\right) +   
    c\, \sin\left( \gamma_{m, k}\, t\right) \right] & \mbox{ for } 0 \leq \beta < 2 \sqrt{\left| \lambda_{m, k} \right|} \mbox{ and } k\in \left\{ 1, \ldots, \lfloor \frac{n}{2} \rfloor \right\} \\
    e^{r_+ t} + c\, \left( e^{r_-t} - e^{r_+ t} \right) & \mbox{ for } \beta \geq 2 \sqrt{\left| \lambda_{m, k} \right|} \mbox{ and } k\in \left\{ 0, \ldots, \lfloor \frac{n}{2} \rfloor \right\}\\ 
   \end{cases}
\end{equation}
such that $r_\pm = - \frac{\beta}{2} \pm \sqrt{\left( \frac{\beta}{2} \right)^2 + \lambda_{m, k}}, \ \gamma_{m, k} = \sqrt{ \left| \left( \frac{\beta}{2} \right)^2 + \lambda_{m, k}\right|}$ and
$c \in \mathbb{R}, c_1, c_2 \in \mathbb{C}$ are constants.  
   
 The higher codimension polygon $\vec{X}(t)$ is the image of a linear transformation of a regular polygon $P_k = (c_k\  s_k)$ with $n$ vertices in $\mathbb{R}^2,$ with linear transformation $T:\mathbb{R}^2 \to \mathbb{R}^p$ given by
   \begin{equation}\label{eqn:linear_transform}
T(x,y) = (x \ \ y)\begin{bmatrix} 
    \alpha_1 & \cdots & \alpha_p\\
    \beta_1 & \cdots & \beta_p\\
    \end{bmatrix}.
\end{equation} 
\end{prop}

\begin{proof}
    Suppose $\vec{X}(t) = g(t)\vec{X}_0$ for differentiable scaling function $g: \mathbb{R} \rightarrow \mathbb{R}.$ Following the same process as in the proof of Proposition \ref{prop:selfsimilar} we find
    \begin{equation*}
        \left(\frac{g"(t)}{f(t)} + \beta \frac{g'(t)}{g(t)}\right)\vec{X}_0 = (-1)^{m+1}M^m\vec{X}_0
    \end{equation*} such that since the right hand side of the above equation is independent of $t,$ we have
    \begin{equation*}
        g"(t) + \beta g'(t) - Cg(t) = 0.
    \end{equation*}

This results in the following
\begin{equation}\label{eqn:highcodim_Csolve}
    \left[(-1)^{m+1}M^m - CI_n\right]\vec{X}_0 = 0_{n\times p.}
\end{equation}
For non-zero $\vec{X}_0$ we therefore require $C$ to be the eigenvalues of the matrix $(-1)^{m+1}M^m$ such that $C= \lambda_{m,k}$ for $k \in \{0,1,\ldots, \lfloor\frac{n}{2}\rfloor\}.$ The corresponding eigenvectors are $P_k=(\vec{c}_k\ \vec{s}_k)$ and $P_{n-k}=(\vec{c}_k\ -\vec{s}_k),$ noting $\vec{s}_{0} = \vec{0}$ and $\vec{s}_{\frac{n}{2}} = \vec{0}$ for $n$ even.  Applying a linear transformation as given in (\ref{eqn:linear_transform}), produces a polygon in $\mathbb{R}^p$ given by
\begin{equation*}
    \vec{X}_0  = (\vec{c}_k\ \vec{s}_k)
            \begin{bmatrix} 
            \alpha_{1} & \cdots & \alpha_{p} \\
            \beta_{1} & \cdots & \beta_{p}\\
            \end{bmatrix},
\end{equation*} and furthermore this polygon solves  (\ref{eqn:highcodim_Csolve}), where $\alpha_{i}$ and $\beta_{i}$ for all $i = 1,\ldots, p$ are any real coefficients. 

Similar to the plane polygon case, solving equation (\ref{eqn:highcodim_Csolve}) for $C=\lambda_{m,k},$ $\beta$ and the required initial condition $g(0) =1$ gives us the expressions for $g(t)$ as stated in (\ref{eqn:selfsim_scale_cases}).

\end{proof}

\begin{prop}\label{lem:selfsim_rotate_highcodim}
    Consider the family $\vec{X}(t)$ of polygons in $\mathbb{R}^p$ such that \begin{equation}\label{eqn:rotateselfsim_highcodim}
    \vec{X}(t) = \vec{X}_0R(t),
\end{equation}
where $R: \mathbb{R} \rightarrow SO(p)$ represents a time-dependent $p\times p$ rotation of the polygon $\vec{X}_0$ in $\mathbb{R}^p$ such that $R(0) = I_p.$ 

If $\vec{X}(t)$ satisfies (\ref{eqn:polyflow}) for some $\beta>0$ and all $t$ then $R(t) \equiv I_p$ and $\vec{X}_0$ corresponds to a point in $\mathbb{R}^p,$ that is the only self similar solution by rotation in this case is the trivial solution. 

If $\vec{X}(t)$ satisfies (\ref{eqn:polyflow}) with $\beta=0$ for all $t$ then pure rotations are possible.

\end{prop}

\begin{proof}

We first consider the case where $\beta >0.$ We also consider the explicit expression for the solutions for (\ref{eqn:polyflow}) stated and proven in Theorem \ref{thm:highcodim_planedamp}. This solution is given in (\ref{E:highcodim_flowsolndamp}) where 
\begin{equation*}
   \vec{X}_0 = \sum_{k=1}^{\lfloor \frac{n}{2} \rfloor}(\vec{c}_k\ \vec{s}_k)
   \begin{bmatrix} 
           \alpha^0_{1k} & \cdots & \alpha^0_{pk}\\
          \beta^0_{1k} & \cdots & \beta^0_{pk}\\
      \end{bmatrix},
\end{equation*}
and 
\begin{equation}\label{E:highcodim_rotateprop}
\vec{X}(t) = (\vec{\alpha}_{10}(t) \cdots \vec{\alpha}_{p0}(t)) + \sum_{k=1}^{\lfloor \frac{n}{2}\rfloor}(\vec{c}_k\ \vec{s}_k)
\begin{bmatrix} 
           \alpha_{1k}(t) & \cdots & \alpha_{pk}(t)\\
          \beta_{1k}(t) & \cdots & \beta_{pk}(t)\\
      \end{bmatrix},
  \end{equation}
where
    $$\alpha_{i0}(t) = \alpha_{i0}^0 + \frac{a_{i0}}{\beta} - \frac{a_{i0}}{\beta}e^{-\beta t},$$ for each $i=1,\ldots, p$ and $a_{i0}$ are arbitrary constants. The $\alpha_{ik}(t)$ and $\beta_{ik}(t)$ are given in (\ref{E:alphaik}) and (\ref{E:betaik}) also in the statement for Theorem \ref{thm:highcodim_planedamp}. T

We require $\vec{X}_0R(t),$ where $R(t)$ is a rotation function, to be of the form given in (\ref{E:highcodim_rotateprop}) above to satisfy (\ref{eqn:polyflow}). The expressions for $\alpha_{ik}(t)$ and $\beta_{ik}(t)$ all include an exponential decaying term $e^{-\frac{b}{2}t}$ and so we must have $a_{i0}=0$, $\alpha_{ik}(t)=0$ and $\beta_{ik}(t)=0$ for all $i=1,\ldots, p$ and $k = 1,\ldots, \lfloor \frac{n}{2} \rfloor.$ This gives

\begin{equation*}
    \vec{X}_0 = (\vec{\alpha}_{10}(t) \cdots \vec{\alpha}_{p0}(t)),
\end{equation*}
and so our initial polygon is a point in $\mathbb R^p$ and we have only the trivial solution for $\vec{X}(t)$ in the $\beta>0$ case. 

For the case where $\beta =0$ we give an example of pure rotation self-similar solutions, without classifying all possible solutions of this type.

We consider the skew symmetric matrix $S: \mathbb{R} \rightarrow \mathfrak{so}(p)$ such that for the exponential map, $\exp:\mathfrak{so}(p) \rightarrow SO(p),$ from the set of skew symmetric $p\times p$ matrices to the set of $p\times p$ rotation matrices, we have $\exp(S(t)) = R(t).$ The second derivative of $R$ is therefore given by

 \begin{equation*}\label{eqn:rotation_deriv_highcodim}
        \frac{d^2R}{dt^2}= \frac{d^2S}{dt^2}R(t) + \left(\frac{dS}{dt}\right)^2R(t).
    \end{equation*}
    
For $\vec{X}(t) = \vec{X}_0R(t)$ to satisfy (\ref{eqn:polyflow}) where $\beta=0$ we have
\begin{equation*}
    \vec{X}_0\frac{d^2R}{dt^2}\ = (-1)^{m+1}M^m\vec{X}_0R(t),
\end{equation*}
which in terms of the skew symmetric matrix and with rearrangement implies
\begin{equation}\label{E:highcodim_scalar}
\vec{X}_0\left[\frac{d^2S}{dt^2} + \left(\frac{dS}{dt}\right)^2\right] = (-1)^{m+1}M^m\vec{X}_0.
\end{equation}

For (\ref{E:highcodim_scalar}) to be true for all $t$ then $\frac{d^2S}{dt^2} + \left(\frac{dS}{dt}\right)^2$ must be a constant matrix we denote as $B,$ and 
\begin{equation}\label{eqn:rotation_highcodim_scalar}
\vec{X}_0B  = (-1)^{m+1}M^m\vec{X}_0.
\end{equation}

We can consider rotations in a plane given by orthogonal axes $x_i$ and $ x_j,$ for $i,j \in \{1,\ldots, p\}$ and $i\neq j,$ where remaining $p-2$ axes are invariant. We denote this rotation as $R_{x_i,x_j}(f_{ij}(t))$ where the differentiable function $f_{ij}:\mathbb{R} \rightarrow \mathbb{R}$ is the angle of rotation at time $t$ in the $x_i-x_j$ plane and where $f_{ij}(0) = 0.$ Furthermore $R_{x_i,x_j}(f_{ij}(t)) = \exp(S_{x_i,x_j}(f_{ij}(t)))$ for skew symmetric matrix $S_{x_i,x_j}(f_{ij}(t)).$  This matrix $S_{x_i,x_j}(f_{ij}(t))= [s_{ab}: a,b = 1,\ldots, p]$ has entries given by $s_{ij} = -s_{ji} = -f_{ij}(t),$ or $s_{ij} = -s_{ji} = f_{ij}(t),$  and all other entries are equal to zero. We also write $S_{x_i,x_j}(f_{ij}(t)) = f_{ij}(t)S_{x_i,x_j}(1)$ and note that 

$$\frac{d}{dt}S_{x_i,x_j}(f_{ij}(t)) = f'_{ij}(t)S_{x_i,x_j}(1) \mbox{ and } \frac{d^2}{dt^2}S_{x_i,x_j}(f_{ij}(t)) = f''_{ij}(t)S_{x_i,x_j}(1).$$ Denoting by $I_{ij,p}$ the $p \times p$ matrix with $1$s in the $(i,i)$th and $(j,j)$th diagonal entries, and zeroes elsewhere, we have  $(S_{x_i,x_j}(1))^2 = -I_{ij,p}.$

Following from (\ref{E:highcodim_scalar}), for a rotation on the $x_i - x_j$ sub-plane, the matrix 
$$[f''_{ij}(t)S_{x_i,x_j}(1) + (f'_{ij}(t))^2(S_{x_i,x_j}(1))^2]$$
has $-(f'_{ij}(t))^2$ in the $(i,i)$th and $(j,j)$th diagonal entries, $-f''_{ij}(t)$ in the $(i,j)$th entry and $f''_{ij}(t)$ in the $(j,i)$th entry (for $i<j$), and all other entries are zero. 

Since this matrix is constant, we must have $f'_{ij}(t) = b,$ for some constant $b$ and so $f''_{ij}(t) = 0$
 for all $t.$ Therefore we have
\begin{equation}\label{E:betazero_rotation}
    -b^2\vec{X}_0 I_{ij,p} = (-1)^{m+1}M^m\vec{X}_0.
\end{equation}
The matrix $\vec{X}_0 I_{ij,p}$ has zero coordinate vectors at each position, except for the $i$th and $j$th coordinate vectors. That is, for $\vec{X}_0 I_{ij,p} = (\vec{x}_1 \cdots \vec{x}_p),$ $\vec{x}_k = \vec{0}$ for all $k \neq i,j.$ We denote this matrix as $\vec{X}_{0,ij}.$

For any $k \in \{1,\ldots, \lfloor \frac{n}{2}\rfloor\},$ consider an initial polygon of 
\begin{equation*}
    \vec{X}_0 = \vec{X}_{0,ij} = (\vec{c}_k\ \vec{s}_k) \begin{bmatrix} 
    \alpha_{1} & \cdots & \alpha_{p}\\
    \beta_{1} & \cdots & \beta_{p}\\
    \end{bmatrix}, 
\end{equation*}
where $\alpha_l = \beta_l = 0$ for all $l \neq i,j,$ and is therefore the image of a regular polygon $P_k = (\vec{c}_k\ \vec{s}_k)$ from $\mathbb R^2$ mapped to $\mathbb R^p$ by a linear transformation, that is planar in the $x_i - x_j$ sub-plane of $\mathbb R^p.$
We note that for this expression of $\vec{X}_0$ which corresponds to the eigenvectors $P_k,$ we have
\begin{equation*}
    (-1)^{m+1}M^m\vec{X}_0 = \lambda_{m,k}\vec{X}_0,
\end{equation*}
for any $k \in \{1,\ldots, \lfloor \frac{n}{2}\rfloor\}.$

From (\ref{E:betazero_rotation}) we therefore have
$$-b^2\vec{X}_0 = -b^2\vec{X}_{0,ij} = (-1)^{m+1}M^m\vec{X}_0,$$
and so $b= \pm \sqrt{-\lambda_{m,k}}$ for $k \in \{0,1,\ldots, n-1\}$ such that $f_{ij}(t) = \pm \sqrt{-\lambda_{m,k}}t.$ The rotation in the $x_i-x_j$ plane is therefore given by $R_{x_i,x_j}(\pm \sqrt{-\lambda_{m,k}}t).$ 
This demonstrates that polygons in $\mathbb R^p$ given by affine images of regular polygons $P_k$ in $\mathbb R^2$ such that they are planar polygons in main sub-planes of $\mathbb R^p$, are self similar solutions by rotation. We suspect that there are also other rotation self-similar solutions for general rotations. 
\end{proof}

\begin{remark}
If the rotation $R(t)$ is acting on a subspace that the polygon $\vec{X}(t)$ is not in, then $\vec{X}(t) = \vec{X}_0R(t) = \vec{X}_0.$ In this case for $\beta\geq 0,$ the only solution is the trivial solution $\vec{X}_0$ where $\vec{X}_0 = (\vec{a}_1, \ldots, \vec{a}_p)$ for real constants $a_i$, $i=1,\ldots, p.$

\end{remark}

We complete the discussion of higher co-dimension self-similar solutions by stating there are no non-trivial solutions by translation. 

\begin{prop}
    Consider the family of polygons with $n$ vertices in $\mathbb{R}^p$ , $\vec{X}(t),$ such that \begin{equation}\label{eqn:translateselfsim}
    \vec{X}(t) = \vec{X}_0 + \vec{h}(t),
\end{equation}
where $\vec{h}(t)$ is a $n\times p$ matrix that represents the translation of the polygon $\vec{X}_0$ and $\vec{h}(0) = 0_{n\times p}.$ If $\vec{X}(t)$ satisfies (\ref{eqn:polyflow}) for all $t$ then each vertex of $\vec{X}_0$ is the same fixed point in $\mathbb{R}^p.$ That is, there are no non-trivial self-similar solutions by translation under (\ref{eqn:polyflow}).  
\end{prop}

\begin{proof}
The proof follows the same process as that for Proposition \ref{prop:selfsim_plane_translate}. In the higher codimension case we find 
$\vec{h}"(t) + \beta \vec{h}'(t) = \vec{C} = (\vec{c}_1,\ldots, \vec{c}_p)$ for real constants $c_i,$  $i = 1,\ldots p.$
Then from Lemma \ref{lem:constant_vec} and Lemma \ref{lem:nullspace_matrix}, the only solution to 
$$(-1)^{m+1}M^m\vec{X}_0 = \vec{C}$$ is if $\vec{X}_0 = (\vec{a}_1,\ldots, \vec{a}_p)$ for real constants $a_1,\ldots, a_p.$ 
  \end{proof}

\subsection{Higher codimension solutions for general initial data}\label{sec:highcodim_solutions}

A similar result to Theorems \ref{thm:planenodamp} and \ref{thm:planedamp} holds for polygons in higher codimensions. 

\begin{theorem}\label{thm:highcodim_nodamp}
   Given an initial polygon $\vec{X}_0$ with $n$ vertices in $\mathbb{R}^p$ such that it is expressed as in (\ref{eqn:highcodim_polygons}), and any $m\in \mathbb{N}$. The equation \eqref{eqn:polyflow} with $\beta =0$ and initial data $\vec{X}\left(0 \right) = \vec{X}_0$ has a unique solution $\vec{X}(t)$ for all time given by 

   \begin{equation*}
    \vec{X}(t)
             =  \left( \alpha_{10}(0) + a_{10}t\  \cdots\ \alpha_{p0}(0) + a_{p0}t \right)\vec{c}_0
             +
             \sum_{k=1}^{\lfloor \frac{n}{2} \rfloor}(\vec{c}_k\ \vec{s}_k)
            \begin{bmatrix} 
            \alpha_{1k}(t) & \cdots & \alpha_{pk}(t)\\
            \beta_{1k}(t) & \cdots & \beta_{pk}(t)\\
            \end{bmatrix},
\end{equation*}

where $$\alpha_{ik}(t) = \alpha_{ik}(0)\cos\left(\sqrt{-\lambda_{m,k}}t\right) + a_{1k}\sin\left(\sqrt{-\lambda_{m,k}}t\right)$$
and 
$$\beta_{ik}(t) = \beta_{ik}(0)\cos\left(\sqrt{-\lambda_{m,k}}t\right) + b_{1k}\sin\left(\sqrt{-\lambda_{m,k}}t\right)$$

where $a_{ik}$ and $b_{ik}$ are arbitrary constants for $i=1,\ldots, p$ and $k = 0,1,\ldots, \lfloor \frac{n}{2} \rfloor.$

\end{theorem}

\begin{proof}

We follow a similar argument to Theorem \ref{thm:planenodamp} and the process given for the parabolic flow cases in  \cite{chow2007semidiscrete} for $m=1$ and in \cite{MM24} for general $m$.  
As given in the setup above and in (\ref{eqn:highcodim_polygons}), we consider a polygon given by
\begin{equation*}
    \vec{X}(t)  = (\vec{x}_1(t)\  \cdots\  \vec{x}_p(t))
             = \sum_{k=0}^{\lfloor \frac{n}{2} \rfloor}(\vec{c}_k\ \vec{s}_k)
            \begin{bmatrix} 
            \alpha_{1k}(t) & \cdots & \alpha_{pk}(t)\\
            \beta_{1k}(t) & \cdots & \beta_{pk}(t)\\
            \end{bmatrix}.
\end{equation*}

The second derivative is given by

\begin{equation*}
    \frac{d^2\vec{X}}{dt^2}  = \frac{d^2}{dt^2}(\vec{x}_1(t)\  \cdots\  \vec{x}_p(t))
             = \sum_{k=0}^{\lfloor \frac{n}{2} \rfloor}(\vec{c}_k\ \vec{s}_k)
            \begin{bmatrix} 
             \alpha''_{1k}(t) & \cdots & \alpha''_{pk}(t)\\
            \beta''_{1k}(t) & \cdots & \beta''_{pk}(t)\\
            \end{bmatrix}.
\end{equation*}

Since the polygon $\vec{X}(t)$ satisfies (\ref{eqn:polyflow}) with $\beta = 0$ then we have 

\begin{multline*}
     \sum_{k=0}^{\lfloor \frac{n}{2} \rfloor}(\vec{c}_k\ \vec{s}_k)
            \begin{bmatrix} 
            \alpha''_{1k}(t) & \cdots & \alpha''_{pk}(t)\\
            \beta''_{1k}(t) & \cdots & \beta''_{pk}(t)\\
            \end{bmatrix}
    = 
(-1)^{m+1}M^m\sum_{k=0}^{\lfloor \frac{n}{2} \rfloor}(\vec{c}_k\ \vec{s}_k)
            \begin{bmatrix} 
            \alpha_{1k}(t) & \cdots & \alpha_{pk}(t)\\
            \beta_{1k}(t) & \cdots & \beta_{pk}(t)\\
            \end{bmatrix}
   \\
   =
\sum_{k=0}^{\lfloor \frac{n}{2} \rfloor}\lambda_{m,k}(\vec{c}_k\ \vec{s}_k)
            \begin{bmatrix} 
            \alpha_{1k}(t) & \cdots & \alpha_{pk}(t)\\
            \beta_{1k}(t) & \cdots & \beta_{pk}(t)\\
            \end{bmatrix}.    
\end{multline*}

Therefore $\alpha''_{ik}(t) = \lambda_{m,k}\alpha_{ik}$ and $\beta''_{ik}(t) = \lambda_{m,k}\beta_{ik}$ for all $i=1, \ldots, p$ and $k = 0,\ldots, \lfloor \frac{n}{2}\rfloor.$

For $k=0$ we have $\lambda_{m,0} = 0$ and so 
$$\alpha_{i0}(t) = \alpha_{i0}(0) + a_{i0}t \text{ and } \beta_{i0}(t) = \beta_{i0}(0) + b_{i0}t,$$ for $i = 1, \ldots, p$ and where $u_{i0}, v_{i0}$ are constants. 

For $k = 1, \ldots, \lfloor \frac{n}{2} \rfloor,$ we have

$$\alpha_{ik}(t) = \alpha_{ik}(0)\cos\left(\sqrt{-\lambda_{m,k}}t\right) + a_{ik}\sin\left(\sqrt{-\lambda_{m,k}}t\right),$$
and similarly
$$\beta_{ik}(t) = \beta_{ik}(0)\cos\left(\sqrt{-\lambda_{m,k}}t\right) + b_{ik}\sin\left(\sqrt{-\lambda_{m,k}}t\right),$$
for constants $a_{ik}, b_{ik},$ and as such the result follows. 
\end{proof}

A damping coefficient $\beta>0$ in \eqref{eqn:polyflow} causes the polygon to shrink to a point in $\mathbb R^p$. Under appropriate rescaling, the solution converges to a polygon in $\mathbb R^p$ that is the affine image of a regular polygon in the plane. 

\begin{theorem}\label{thm:highcodim_planedamp}
Given an initial polygon $\vec{X}_0$ with $n$ vertices in $\mathbb{R}^p$ as given by
\begin{equation*}
   \vec{X}_0 = \sum_{k=1}^{\lfloor \frac{n}{2} \rfloor}(\vec{c}_k\ \vec{s}_k)
   \begin{bmatrix} 
           \alpha^0_{1k} & \cdots & \alpha^0_{pk}\\
          \beta^0_{1k} & \cdots & \beta^0_{pk}\\
      \end{bmatrix},
\end{equation*}
and any $m\in \mathbb{N}$, the equation \eqref{eqn:polyflow} with $\beta >0$ constant and initial data $\vec{X}\left(0 \right) = \vec{X}_0$ has a unique solution given by
\begin{equation} \label{E:highcodim_flowsolndamp}
\vec{X}(t) = (\vec{\alpha}_{10}(t) \cdots \vec{\alpha}_{p0}(t)) + \sum_{k=1}^{\lfloor \frac{n}{2}\rfloor}(\vec{c}_k\ \vec{s}_k)
\begin{bmatrix} 
           \alpha_{1k}(t) & \cdots & \alpha_{pk}(t)\\
          \beta_{1k}(t) & \cdots & \beta_{pk}(t)\\
      \end{bmatrix},
  \end{equation}

    where
    $$\alpha_{i0}(t) = \alpha_{i0}^0 + \frac{a_{i0}}{\beta} - \frac{a_{i0}}{\beta}e^{-\beta t},$$ for each $i=1,\ldots, p$ and where $a_{i0}$ are arbitrary constants. 

    Each $\alpha_{ik}(t)$ is given by
    \begin{equation} \label{E:alphaik}
  \alpha_{ik}\left( t\right) = \begin{cases}
        \alpha^0_{ik} e^{r_{+m, k} t} + a_{ik} (e^{r_{-m, k} t} - e^{r_{+m, k} t}) &\mbox{ for } \left| \lambda_{m, k} \right| < \frac{\beta^2}{4}\\
        (\alpha^0_{ik} + a_{ik} t) e^{-\frac{\beta}{2} t} &\mbox{ if } \lambda_{m, k} = - \frac{\beta^2}{4}\\
        e^{-\frac{\beta}{2} t} \left[ \alpha^0_{ik} \cos \left( \gamma_{m, k} t \right) + a_{ik} \sin\left( \gamma_{m, k} t\right) \right],     &\mbox{ for } \left| \lambda_{m, k} \right| > \frac{\beta^2}{4}    
  \end{cases}
  \end{equation}
  and similarly for $\beta_{ik}(t),$

\begin{equation} \label{E:betaik}
  \beta_{ik}\left( t\right) = \begin{cases}
        \beta^0_{ik} e^{r_{+m, k} t} + b_{ik} (e^{r_{-m, k} t} - e^{r_{+m, k} t}) &\mbox{ for } \left| \lambda_{m, k} \right| < \frac{\beta^2}{4}\\
        (\beta^0_{ik} + b_{ik} t) e^{-\frac{\beta}{2} t} &\mbox{ if } \lambda_{m, k} = - \frac{\beta^2}{4}\\
        e^{-\frac{\beta}{2} t} \left[ \beta^0_{ik} \cos \left( \gamma_{m, k} t \right) + b_{ik} \sin\left( \gamma_{m, k} t\right) \right],     &\mbox{ for } \left| \lambda_{m, k} \right| > \frac{\beta^2}{4}    
  \end{cases}
  \end{equation}
  where
  $$r_{\pm m, k} = - \frac{\beta}{2} \pm \sqrt{\left( \frac{\beta}{2}\right)^2 + \lambda_{m, k}} $$
  and
  $$\gamma_{m, k} = \sqrt{-\lambda_{m, k} - \left( \frac{\beta}{2}\right)^2} \mbox{.}$$
The constants $a_{ik}$ and $b_{ik}$ are arbitrary constants for $i=1,\ldots, p$ and $k = 1,\ldots, \lfloor \frac{n}{2} \rfloor.$

The solution exists for all time and converges exponentially to a point in $\mathbb R^p$.  If the dominant eigenvalue $\lambda_{m,k}$ satisfies $|\lambda_{m,d}| \leq \frac{\beta^2}{4},$ then under appropriate rescaling, the solution is asymptotic as $t\rightarrow \infty$ to a planar polygon in $\mathbb R^p$ with $n$ vertices which is an affine image of a regular polygon in $\mathbb R^2.$  If $|\lambda_{m,d}|> \frac{\beta^2}{4},$ the solution exhibits continued oscillating behaviour as it shrinks to a point.

\end{theorem}

\begin{proof}
   Since the polygon $\vec{X}(t)$ is to satisfy (\ref{eqn:polyflow}) with $\beta>0,$ we obtain

\begin{multline*}
     \frac{d^2\vec{X}}{dt^2} + \beta \frac{d\vec{X}}{dt} = \sum_{k=0}^{\lfloor \frac{n}{2} \rfloor}(\vec{c}_k\ \vec{s}_k)
            \begin{bmatrix} 
            \alpha''_{1k}(t) +\beta \alpha'_{1k}(t) & \cdots & \alpha''_{pk}(t)+\beta \alpha'_{pk}(t)\\
            \beta''_{1k}(t) +\beta \beta'_{1k}(t)& \cdots & \beta''_{pk}(t)+\beta \beta'_{pk}(t)\\
            \end{bmatrix}
            \\
   =
   (-1)^{m+1}M^m\sum_{k=0}^{\lfloor \frac{n}{2} \rfloor}(\vec{c}_k\ \vec{s}_k)
            \begin{bmatrix} 
            \alpha_{1k}(t) & \cdots & \alpha_{pk}(t)\\
            \beta_{1k}(t) & \cdots & \beta_{pk}(t)\\
            \end{bmatrix}
            =
\sum_{k=0}^{\lfloor \frac{n}{2} \rfloor}\lambda_{m,k}(\vec{c}_k\ \vec{s}_k)
            \begin{bmatrix} 
            \alpha_{1k}(t) & \cdots & \alpha_{pk}(t)\\
            \beta_{1k}(t) & \cdots & \beta_{pk}(t)\\
            \end{bmatrix}.    
\end{multline*}

Therefore we are left to solve
$$\alpha''_{ik}(t) +\beta \alpha'_{ik}(t) = \lambda_{m,k}\alpha_{ik}(t)$$
and 
$$\beta''_{ik}(t) +\beta \beta'_{ik}(t) = \lambda_{m,k}\beta_{ik}(t),$$
for $i =1,\ldots, p$ and $k = 0, \ldots \lfloor \frac{n}{2}\rfloor.$

This leads to expression for $\alpha_{i0}(t)$ as stated in the theorem, and the cases for $\alpha_{ik}(t)$ and $\beta_{ik}(t)$ as given in (\ref{E:alphaik}) and (\ref{E:betaik}).

With the same reasoning as given for the planar case in Theorem \ref{thm:planedamp}, as $t\to \infty$ then each vertex of the polygon will converge to the same point in $\mathbb R^p$ given by $\left(\alpha_{10}^0 + \frac{a_{10}}{\beta},  \cdots,  \alpha_{p0}^0 + \frac{a_{p0}}{\beta}\right).$

To determine the limiting shape of the polygon as it shrinks, again we consider an appropriate rescaling and translation of $\vec{X}(t)$ given by
\begin{equation}\label{eqn:highcodim_shape}
    \vec{Y}(t) = g(t)\left(\vec{X}(t) - (\vec{\alpha}_{10}(t) \cdots \vec{\alpha}_{p0}(t))\right),
    \end{equation}

and where the scaling function is chosen based on the chosen damping term $\beta$ and the relationship of eigenvalues $\lambda_{m,k}$ to this damping term, that then determine the expression for the coefficients $\alpha_{ik}(t)$ and $\beta_{ik}(t).$

If the dominant eigenvalue $\lambda_{m,1}$ satisfies $|\lambda_{m,1}| < \frac{\beta^2}{4},$ then we choose scaling function $g(t) = e^{-r_{+m,1}t}.$ By similar reasoning as in Theorem \ref{thm:planedamp}, we find that

\begin{equation*}
    \lim_{t\to \infty}\vec{Y}(t) = (\vec{c}_1\ \vec{s}_1) 
    \begin{bmatrix}
        \alpha_{11}^0 - a_{11} & \cdots &\alpha_{p1}^0 - a_{p1}\\
        \beta_{11}^0 - b_{11} & \cdots &\beta_{p1}^0 - b_{p1}.\\
    \end{bmatrix}.
\end{equation*}

We have $P_1 = (\vec{c}_1\ \vec{s}_k)$ which is a regular polygon in plane $\mathbb R^2.$ Therefore the limit of $\vec{Y}(t)$ is the polygon $P_1$ where each vertex is mapped to to a vertex in $\mathbb R^p$ by the linear map $T_1:\mathbb R^2 \to \mathbb R^p$ given by

\begin{equation*}
T_1(x,y) = (x \ \ y)\begin{bmatrix}
        \alpha_{11}^0 - a_{11} & \cdots &\alpha_{p1}^0 - a_{p1}\\
        \beta_{11}^0 - b_{11} & \cdots &\beta_{p1}^0 - b_{p1}.\\
    \end{bmatrix}.
\end{equation*}

$T_1$ is a linear transformation, and since $P_1$ is in the plane, the image of $P_1$ under $T_1$ is two dimensional and therefore planar.

If $\lambda_{m,1} = -\frac{\beta^2}{4}$ then we choose the scaling factor to be $g(t) = \frac{e^{\frac{\beta}{2}t}}{t}$ in (\ref{eqn:highcodim_shape}). Taking the limit gives

\begin{equation*}
\lim_{t\to\infty}\vec{Y}(t) = (\vec{c}_1\ \vec{s}_1) 
    \begin{bmatrix}
         a_{11} & \cdots & a_{p1}\\
    b_{11} & \cdots & b_{p1}.\\
    \end{bmatrix}.
\end{equation*}

Therefore considering the linear map $T_2:\mathbb R^2 \to \mathbb R^p$ given by

\begin{equation*}
T_2(x,y) = (x \ \ y)\begin{bmatrix}
        a_{11} & \cdots & a_{p1}\\
        b_{11} & \cdots & b_{p1}.\\
    \end{bmatrix}
\end{equation*} we have the same result as above where the limiting shape is given by the image of $P_1$ in $\mathbb R^2$ mapped to $\mathbb R^p$ by $T_2.$

Given $T_1$ and $T_2$ are linear transformations from $\mathbb R^2$, the resulting polygon in $\mathbb R^p$ that is the image of these maps, is planar. 

If $|\lambda_{m,1}|>\frac{\beta^2}{4}$ then we choose the scaling factor $g(t) = e^{\frac{\beta}{2}t}$ for the expression in (\ref{eqn:highcodim_shape}) which gives

\begin{equation*}
    \vec{Y}(t) = \sum_{k=1}^{\lfloor \frac{n}{2}\rfloor}(\vec{c}_k\ \vec{s}_k)
            \begin{bmatrix} 
            \alpha_{1k}^0\cos(\gamma_{m,k}t)+a_{1k}\sin(\gamma_{m,k}t) & \cdots & \alpha_{pk}^0\cos(\gamma_{m,k}t)+a_{pk}\sin(\gamma_{m,k}t)\\
           \beta_{1k}^0\cos(\gamma_{m,k}t) +b_{1k}\sin(\gamma_{m,k}t)& \cdots & \beta_{pk}^0\cos(\gamma_{m,k}t)+b_{pk}\sin(\gamma_{m,k}t)\\
            \end{bmatrix}
\end{equation*}
 and as such, the polygon exhibits continued oscillating behaviour.

In the case of $\alpha_{i1}(0) =0$ and $\beta_{i1}(0) = 0$ for all $i = 1,2,\ldots p,$ we instead scale the polygon by an expression involving $\lambda_{m,d}$ such that $\lambda_{m,d}$ is the next dominant eigenvalue where we have non-zero $\alpha_{id}, \beta_{id}$, and carry out the same process as described above depending on the relationship $\lambda_{m,d}$ has with the damping term $\beta.$

\end{proof}

\begin{remark}
    We may also examine \emph{ancient} solutions of our hyperbolic polyharmonic flows of polygons by considering limits as $t\rightarrow -\infty$ in solution formulae, which always make sense in this setting. In the $\beta>0$ case for planar polygons, we get different outcomes based on the relationships of the least and most dominant eigenvalues with $\beta,$ as well as the arbitrary constants $a_k$ chosen in the expression (\ref{E:flowsolndamp}) and (\ref{E:alphak}). 

    We demonstrate such solutions by following a similar argument as given in Theorem \ref{thm:planedamp} by rescaling and translating $\vec{X}(t)$ as follows
    \begin{equation*}
     \vec{Y}(t) = g(t)(\vec{X}(t) - \vec{a}_0(0)),
\end{equation*} for appropriate rescaling factor $g(t).$

    If we have $|\lambda_{m,\lfloor \frac{n}{2}\rfloor}|\leq \frac{\beta^2}{4}$ and if $a_k = 0$ for all $k \neq  \lfloor \frac{n}{2}\rfloor,$ then we observe asymptotic convergence to the affine transformation of the regular polygon $P_{\lfloor \frac{n}{2}\rfloor}$ as $t\to -\infty.$ We see this by choosing scaling factor $g(t) = e^{-r_{+m,\lfloor \frac{n}{2}\rfloor}t}$ if $|\lambda_{m,\lfloor \frac{n}{2}\rfloor}| < \frac{\beta^2}{4},$ 
    
    and note
    $$\lim_{t\to -\infty} \vec{Y}(t) = \left(\alpha^0_{\lfloor \frac{n}{2}\rfloor}- a_{\lfloor \frac{n}{2}\rfloor}\right)P_{\lfloor \frac{n}{2}\rfloor} + \left(\alpha^0_{n-\lfloor \frac{n}{2}\rfloor}- a_{n-\lfloor \frac{n}{2}\rfloor}\right)P_{n-\lfloor \frac{n}{2}\rfloor}.$$ 

If $|\lambda_{m,\lfloor \frac{n}{2}\rfloor}| = \frac{\beta^2}{4}$ we choose $g(t) = \frac{e^{\frac{\beta}{2}t}}{t}$ and obtain
$$\lim_{t\to -\infty} \vec{Y}(t) = a_{\lfloor \frac{n}{2}\rfloor}P_{\lfloor \frac{n}{2}\rfloor} + a_{n-\lfloor \frac{n}{2}\rfloor}P_{n-\lfloor \frac{n}{2}\rfloor}.$$

When $n$ is even, this will be the straight line interval of $\frac{n}{2}$ overlapping polygon edges. If this original polygon is orthogonal to $P_{\lfloor \frac{n}{2}\rfloor}$ then we consider the next least dominant eigenvalue and get a similar result, provided the remaining arbitrary constants $a_k$ are zero. 

If we have at least $a_1$ or $a_{n-1}$ not equal to zero, where $\lambda_{m,1}$ is the dominant eigenvalue and $|\lambda_{m,1}| \leq \frac{\beta^2}{4},$ then choosing a rescaling factor of $g(t) = e^{-r_{-m,1}t}$ if $|\lambda_{m,1}| < \frac{\beta^2}{4},$ or $g(t) = \frac{e^{\frac{\beta}{2}t}}{t}$ if $|\lambda_{m,1}| = \frac{\beta^2}{4},$ demonstrates asymptotic convergence to an affine transformation of regular polygon $P_1$ as $t\to -\infty,$ given by
\begin{equation*}
    \lim_{t \to -\infty}\vec{Y}(t) = a_1P_1 + a_{n-1}P_{n-1}.
\end{equation*}

A similar result holds for solutions in higher codimension using the scaled polygon 
\begin{equation*}
    \vec{Y}(t) = g(t)\left(\vec{X}(t) - \left(\vec{\alpha}_{10}(0)\ \ldots \ \vec{\alpha}_{p0}(0)\right)\right),
\end{equation*}
and $g(t)$ chosen as described above. Here we need to consider the values of the arbitrary constants $a_{ik}$ and $b_{ik}$ in the general solution given in (\ref{E:highcodim_flowsolndamp}). Whether they are zero or not will determine the limiting behaviour of the polygon as $t\to -\infty$ and whether we have asymptotic convergence to an affine image of $P_1$ or $P_{\lfloor \frac{n}{2} \rfloor }$ in $\mathbb R^p,$ with the observed outcome based on the same reasoning as discussed for the plane polygon case. 

\end{remark}

\section{Semi-discrete geometric flow between polygons} \label{S:Yau}

In their 2009 paper \cite{lin2009evolving}, Lin and Tsai described Yau's curvature difference flow whose objective is to evolve one curve to another, either in finite time or in infinite time, possibly up to an isometry, using a parabolic flow.  In \cite{MM24} we considered a first order in time semi-discrete polyharmonic flow analogue of Yau's curvature difference flow.  Here we consider a second order in time analogue that includes a linear damping term.

\begin{theorem} \label{T:Yau}
Given a target polygon $\vec{Y}$ with $n$ vertices in $\mathbb R^p$, for $p\in \mathbb N, p \geq 2,$ and any fixed $\beta>0$, all solutions $\vec{X}(t)$ to 
the second order semi-discrete Yau difference flow, 
\begin{equation}\label{eqn:Yau_diff_eqn}\tag{$SYDF_m$}
    \frac{d^2\vec{X}}{dt^2} + \beta \frac{d \vec{X}}{dt}= (-1)^{m+1}M^m\left[\vec{X}(t) - \vec{Y}\right]
\end{equation} 
with any initial polygon $\vec{X}(0) =\vec{X}_0$ with $n$ vertices, exists for all time and converges as $t\rightarrow \infty$ to $\vec{Y}$.
\end{theorem}

\begin{remark}
\begin{enumerate}
    \item Recalling the operator $(-1)^{m+1} M^m$ is a higher order linear curvature-type operator, the right hand side of \eqref{eqn:Yau_diff_eqn} is precisely the difference in this type of curvature of $\vec{X}$ as compared with that of $\vec{Y}$.  When $\vec{X} \equiv \vec{Y}$, the right hand side is identically equal to zero and $\vec{X} \equiv \vec{Y}$ is a stationary solution.
    \item Of particular interest is that convergence to $\vec{Y}$ is obtained for \emph{any} initial polygon.
    \item Above we say `all solutions' because the problem as written is underdetermined.  As earlier, to obtain a unique solution we must provide some additional data, for example, an initial `velocity' or a specific polygon through which the solution transitions at a specific time.  Our result says whatever the initial polygon and additional data, we will always have long time existence, and convergence as $t\rightarrow \infty$ to $\vec{Y}$.
    \item If, instead of having the same number of vertices, one of $\vec{X}_0$ or $\vec{Y}$ has fewer vertices, we can simply duplicate vertices or add points on the line segments joining vertices to create initial and target polygons with the same number of `vertices'.  This process is described in more detail in \cite{MM24}.
\end{enumerate}
    
\end{remark}

\begin{proof}
As in \cite{MM24} we set up a \textit{difference polygon}, $\vec{Z}(t) = \vec{X}(t) - \vec{Y}.$  

Write
\begin{equation*}
    \vec{Y} = \sum_{k=0}^{\lfloor \frac{n}{2}\rfloor}(\vec{c}_k\ \vec{s}_k) 
     \begin{bmatrix} 
            y^{(1)}_{1k} & \cdots & y^{(1)}_{pk}\\
            y^{(2)}_{1k} & \cdots & y^{(2)}_{pk}\\
            \end{bmatrix}\\,
\end{equation*}
and 
\begin{equation*}
   \vec{X}_0 = \sum_{k=0}^{\lfloor \frac{n}{2} \rfloor}(\vec{c}_k\ \vec{s}_k)
   \begin{bmatrix} 
           \alpha^0_{1k} & \cdots & \alpha^0_{pk}\\
          \beta^0_{1k} & \cdots & \beta^0_{pk}\\
      \end{bmatrix}.
\end{equation*}

Then 
\begin{equation*}
   \vec{Z}_0 = \sum_{k=0}^{\lfloor \frac{n}{2} \rfloor}(\vec{c}_k\ \vec{s}_k)
   \begin{bmatrix} 
           \alpha^0_{1k} - y^{(1)}_{1k} & \cdots & \alpha^0_{pk} - y^{(1)}_{pk}\\
          \beta^0_{1k} - y^{(2)}_{1k} & \cdots & \beta^0_{pk}- y^{(2)}_{pk}\\
      \end{bmatrix}.
\end{equation*}

Since $\vec{Y}$ is constant, it follows that $\vec{Z}(t)$ satisfies equation \eqref{eqn:polyflow} and therefore, by Theorem \ref{thm:highcodim_planedamp} has solution
\begin{equation*}
\vec{Z}(t) = \sum_{k=0}^{\lfloor \frac{n}{2} \rfloor}(\vec{c}_k\ \vec{s}_k)
   \begin{bmatrix} 
           z^{(1)}_{1k}(t) & \cdots & z^{(1)}_{pk}(t)\\
          z^{(2)}_{1k}(t) & \cdots & z^{(2)}_{pk}(t)\\
      \end{bmatrix},
\end{equation*}

where for each $i = 1,\ldots, p,$ 

\begin{equation}
    z^{(1)}_{i0}(t) = \left(\alpha^0_{i0} - y^{(1)}_{i0}\right) + \frac{a_{i0}}{\beta} - \frac{a_{i0}}{\beta}e^{-\beta t},
\end{equation}
where $a_{i0}$ are arbitrary constants,

and for $k=1, 2, \ldots, n-1$,
\begin{equation}\label{E:yaucoeffs_1}
    z^{(1)}_{ik}(t) = \begin{cases} 
    \left( \alpha_{ik}^0 - y^{(1)}_{ik} \right) e^{r_{-m, k}t} + a_{ik} \left(e^{r_{+m, k} t} - e^{r_{-m, k}t}\right) & \mbox{ for }\left| \lambda_{m, k} \right| < \frac{\beta^2}{4}\\
    \left(  \alpha_{ik}^0 - y^{(1)}_{ik}  + a_{ik}\, t\right) e^{r_{m, k} t}& \mbox{ if } \lambda_{m, k} = - \frac{\beta^2}{4}\\
    e^{-\frac{\beta}{2} t}\left[ \left( \alpha_{ik}^0 - y^{(1)}_{ik} \right) \cos \left( \gamma_{m, k} t\right) + a_{ik} \sin \left( \gamma_{m, k} t\right) \right] & \mbox{ for } \left| \lambda_{m, k} \right| > \frac{\beta^2}{4}
\end{cases}
\end{equation}
and
\begin{equation}\label{E:yaucoeffs_2}
    z^{(2)}_{ik}(t) = \begin{cases} 
    \left( \beta_{ik}^0 - y^{(2)}_{ik} \right) e^{r_{-m, k}t} + b_{ik} \left(e^{r_{+m, k} t} - e^{r_{-m, k}t}\right) & \mbox{ for }\left| \lambda_{m, k} \right| < \frac{\beta^2}{4}\\
    \left(  \beta_{ik}^0 - y^{(2)}_{ik}  + b_{ik}\, t\right) e^{r_{m, k} t}& \mbox{ if } \lambda_{m, k} = - \frac{\beta^2}{4}\\
    e^{-\frac{\beta}{2} t}\left[ \left( \beta_{ik}^0 - y^{(2)}_{ik} \right) \cos \left( \gamma_{m, k} t\right) + b_{ik} \sin \left( \gamma_{m, k} t\right) \right] & \mbox{ for } \left| \lambda_{m, k} \right| > \frac{\beta^2}{4}
\end{cases}
\end{equation},
where 
$$r_{\pm m, k} = - \frac{\beta}{2} \pm \sqrt{\left( \frac{\beta}{2} \right)^2 + \lambda_{m, k}} \mbox{, } \gamma_{m, k} = \sqrt{ - \lambda_{m, k} - \left( \frac{\beta}{2} \right)^2} \mbox{.}$$
The constants $a_{ik}$ and $b_{ik}$ are completely free.
Hence
$$
\vec{X}(t) = \left[ \alpha_{10}^0 + \frac{a_{10}}{\beta} - \frac{a_{10}}{\beta} e^{-\beta\, t} \ \cdots \  \alpha_{p0}^0 + \frac{a_{p0}}{\beta} - \frac{a_{p0}}{\beta} e^{-\beta\, t}\right] \vec{c}_0 + \sum_{k=1}^{\lfloor\frac{n}{2}\rfloor} (\vec{c}_k\ \vec{s}_k)\begin{bmatrix} 
           z^{(1)}_{1k}(t)+y^{(1)}_{1k} & \cdots & z^{(1)}_{pk}(t)+y^{(1)}_{pk}\\
          z^{(2)}_{1k}(t)+y^{(2)}_{1k} & \cdots & z^{(2)}_{pk}(t)+y^{(2)}_{pk}\\
      \end{bmatrix} \mbox{.}
$$
In view of \eqref{E:yaucoeffs_1} and \eqref{E:yaucoeffs_2}, provided $a_{i0}$ is chosen such that 
$$\alpha_{i0}^0 + \frac{a_{i0}}{\beta} = y^{(1)}_{i0}$$ for all $i=1,\ldots, p,$ we will have $\vec{X}(t) \rightarrow \vec{Y}$ as $t\rightarrow \infty$. Otherwise $\vec{X}$ will converge to a translate of $\vec{Y}$ as $t\to \infty.$
\end{proof}

The case of a moving target polygon $\vec{Y}(t)$ may be handled using Green's functions as follows.  Considering polygons in the plane, our equation can now be written as
\begin{equation} \label{E:Yaunonhom}
  \frac{d^2 \vec{X}}{dt^2} + \beta \frac{d\vec{X}}{dt} = \left( -1\right)^{m+1} M^m \left[ \vec{X}(t) - \vec{Y}(t) \right]
  \end{equation}
where $\vec{Y}(t) = \sum_{k=0}^{n-1} y_k(t) P_k$.  

Specifically, in view of Proposition \ref{prop:chow} we may seek a solution to \eqref{E:Yaunonhom} of the form
\begin{equation} \label{E:Yausoln}
  \vec{X}(t) = \sum_{k=0}^{n-1} \alpha_k(t) P_k \mbox{.}
\end{equation}
Then, by linearity of \eqref{E:Yaunonhom}, the coefficients of the evolving polygon $\vec{X}(t)$ satisfy
\begin{equation} \label{E:Yaucoeff}
  \alpha_k''(t) + \beta\, \alpha_k'(t) = \lambda_{m, k} \left[ \alpha_k(t) - y_k(t) \right] \mbox{.}
  \end{equation}
with solution for $k=0$
$$\alpha_0(t) = c_0 e^{-\beta\, t} + d_0$$
with arbitrary constants $c_0, d_0$. The given data will put at least one condition on these constants, so only in special cases will the translation term approach that of $\vec{Y}$.

For $k=1, 2, \ldots, n-1,$ writing as usual
$$r_{\pm m, k} = - \frac{\beta}{2} \pm \sqrt{ \left( \frac{\beta}{2} \right)^2 + \lambda_{m, k}} \mbox{,}$$
we have the independent real solutions to the homogeneous equation
$$\alpha_k^{(1)}(t) = e^{r_{-m, k} t } \mbox{ and } \alpha_k^{(2)}(t) = e^{r_{+m, k} t} \mbox{ for } \left| \lambda_{m, k} \right| < \frac{\beta^2}{4} \mbox{,}$$
$$\alpha_k^{(1)}(t) = e^{r_{m, k} t} \mbox{ and } \alpha_k^{(2)}(t) = t\, e^{r_{m, k} t}\mbox{ if } \lambda_{m, k} = -\frac{\beta^2}{4}$$
and
$$\alpha_k^{(1)}(t) = e^{-\frac{\beta}{2} t } \cos\left( \gamma_{m, k} t\right) \mbox{ and } \alpha_k^{(2)}(t) = e^{-\frac{\beta}{2} t } \sin\left( \gamma_{m, k} t\right) \mbox{ for } \left| \lambda_{m, k} \right| > \frac{\beta^2}{4} \mbox{,}$$
where $\gamma_{m, k} = \sqrt{ - \lambda_{m, k} - \left( \frac{\beta}{2}\right)^2}$.
These allow us to write down a Green's function corresponding to each coefficient function:
$$G_k\left( x, t\right) = \frac{1}{\det W_k\left( t\right)} \left[ \alpha_k^{(1)}( x ) \alpha_k^{(2)}( t ) - \alpha_k^{(1)}( t ) \alpha_k^{(2)}( x) \right]\mbox{,}$$
where
$$W_k\left( t\right) = \left[ \begin{matrix}
    \alpha_k^{(1)}(t) & \alpha_k^{(2)}(t)\\
    \frac{d}{dt} \alpha_k^{(1)}(t) & \frac{d}{dt} \alpha_k^{(2)}(t)\\
\end{matrix} \right] \mbox{.}$$
The solution to the nonhomogeneous equation \eqref{E:Yaucoeff} is then
$$\alpha_k(t) = c_{-k} \alpha_k^{(1)}(t) + c_{+k} \alpha_k^{(2)}(t) + \int_0^t G_k\left( x, t \right) y_k(x) dx \mbox{,}$$
where $c_{\pm k}$ are arbitrary constants.  The solution polygon is then $$\vec{X}(t) = \left[ c_0 e^{-\beta\, t} + d_0 \right] P_0 + \sum_{k=1}^{n-1} \left[ c_{-k} \alpha_k^{(1)}(t) + c_{+k} \alpha_k^{(2)}(t) + \int_0^t G_k\left( x, t \right) y_k(x) dx \right] P_k \mbox{.}$$
The limiting solution will not translate as per $\vec{Y}$ unless $d_0= y_0$ is constant.  In view of the form of the coefficients of the other $P_k$, whether the solution will approach $\vec{Y}(t)$ as $t\rightarrow \infty$ depends precisely on the behaviour of the terms involving the Green's functions.

\begin{figure*}

  \begin{subfigure}[b]{0.31\textwidth}
    \includegraphics[width=\textwidth]{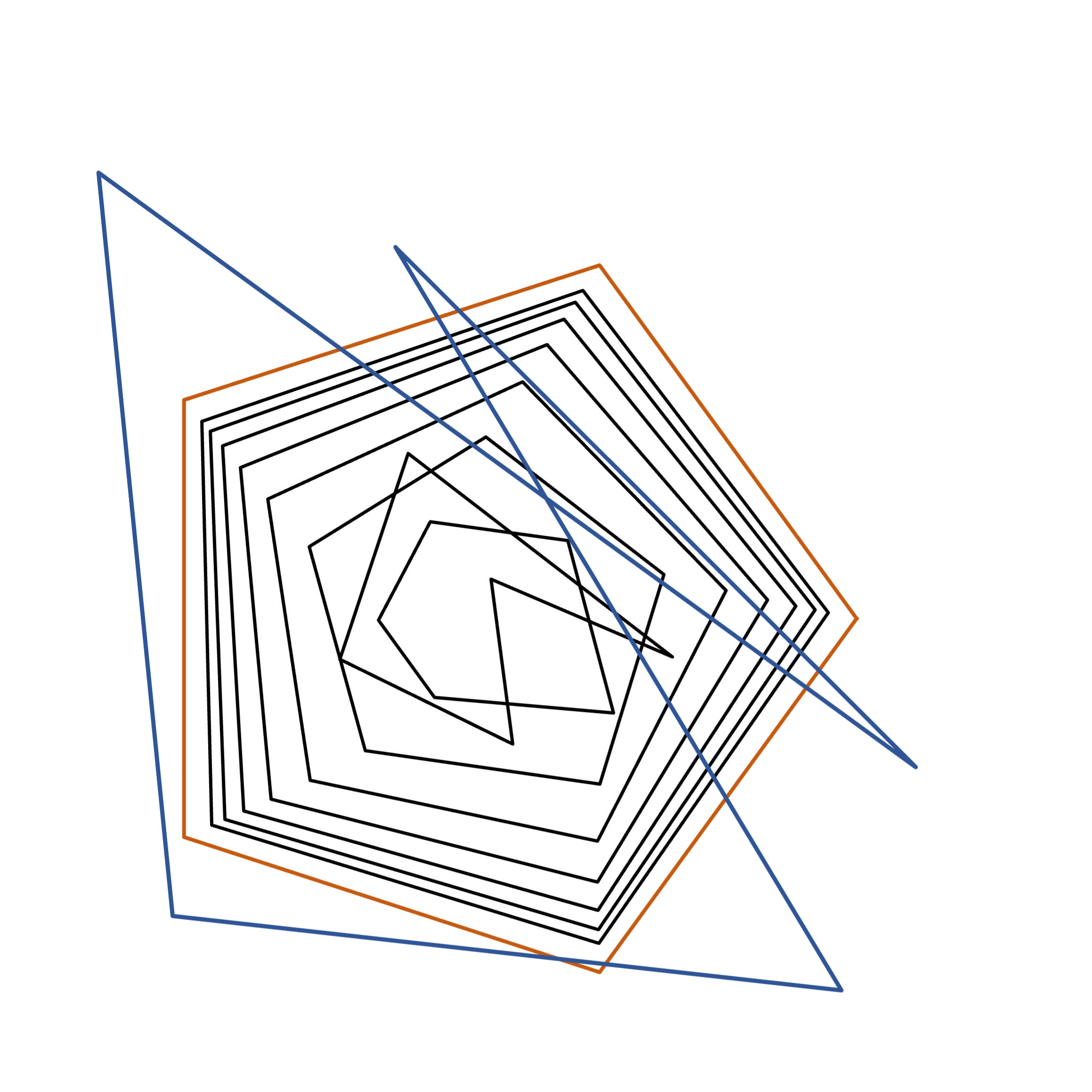}
    \caption{$m=1,\ \beta=4,\ $all $a_k=0$ }
    \label{fig:Yau_1_m=1}
  \end{subfigure}
  \hfill
  \begin{subfigure}[b]{0.31\textwidth}
    \includegraphics[width=\textwidth]{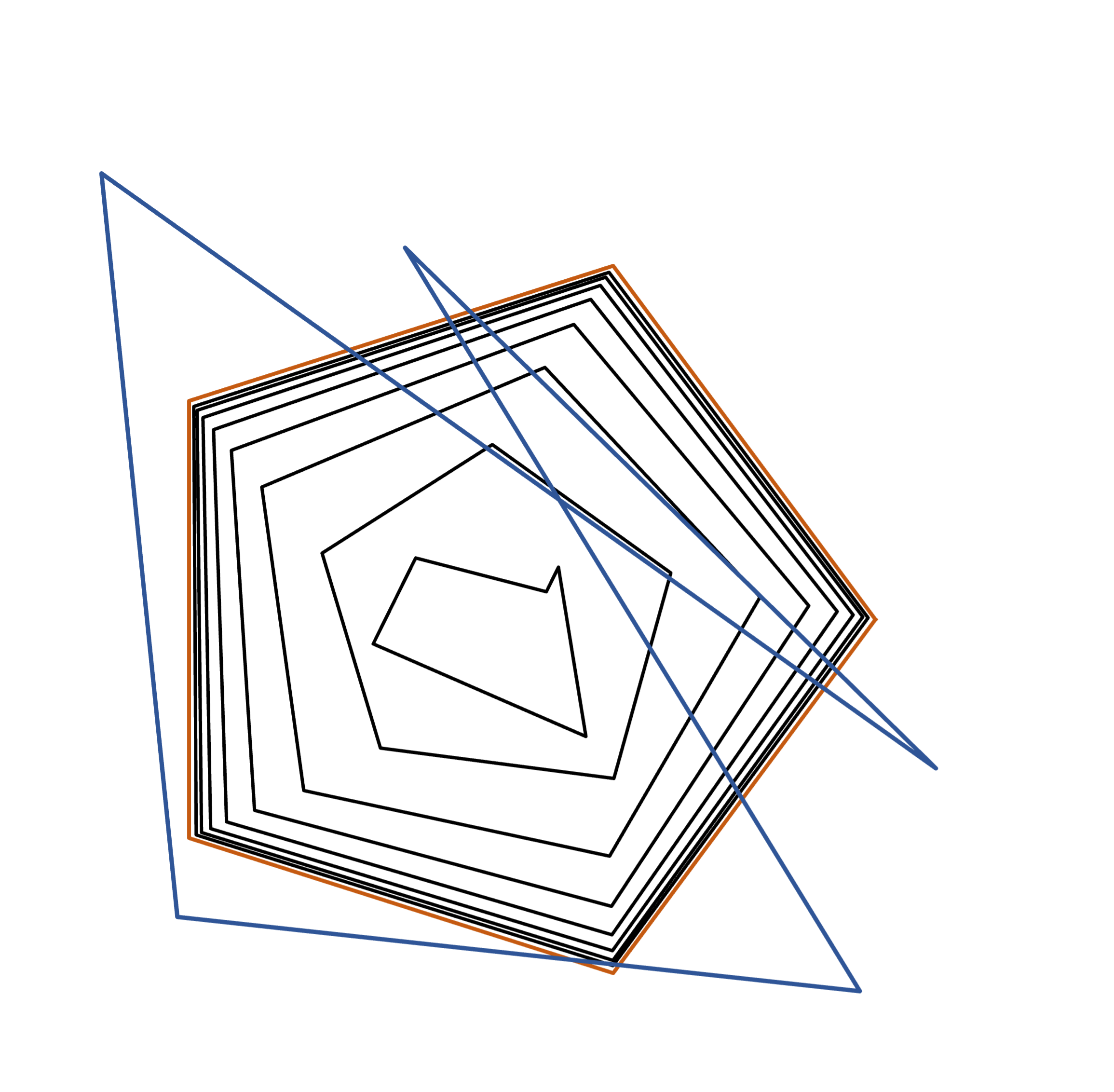}
    \caption{$m=2,\ \beta=4,\ $all $a_k = 0$}
    \label{fig:Yau_2_m=1}
  \end{subfigure}
  \hfill
   \begin{subfigure}[b]{0.31\textwidth}
    \includegraphics[width=\textwidth]{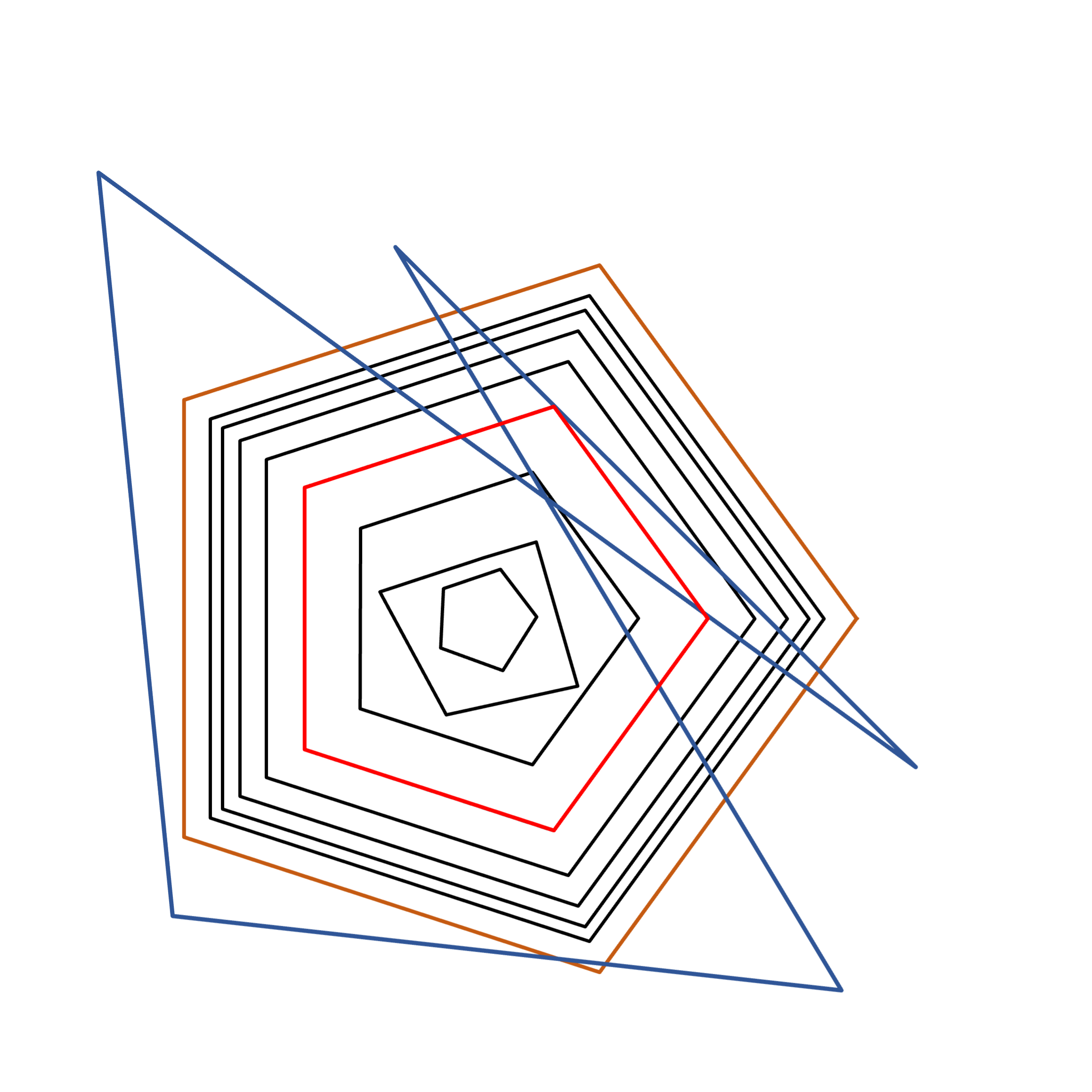}
    \caption{$m=1,\ \beta=4,\ a_k$ non-zero }
    \label{fig:Yau_3_m=1_inter}
  \end{subfigure}
 
  \caption{Different cases of pentagons flowing to regular pentagons under the semi-discrete Yau difference flow. In each case, selected time steps of the evolution are shown superimposed over the initial and target polygons. The initial polygon is given in blue and the target polygon in orange. The target polygon in this case is $5P_1.$ In (c), the arbitrary coefficients $a_k$ are prescribed such that the polygon flows to an intermediate polygon at a particular time, depicted in red. In this case the intermediate polygon is $3P_1$.} \label{fig:Yau_cases}
\end{figure*}

\begin{figure*}

  \begin{subfigure}[b]{0.31\textwidth}
    \includegraphics[width=\textwidth]{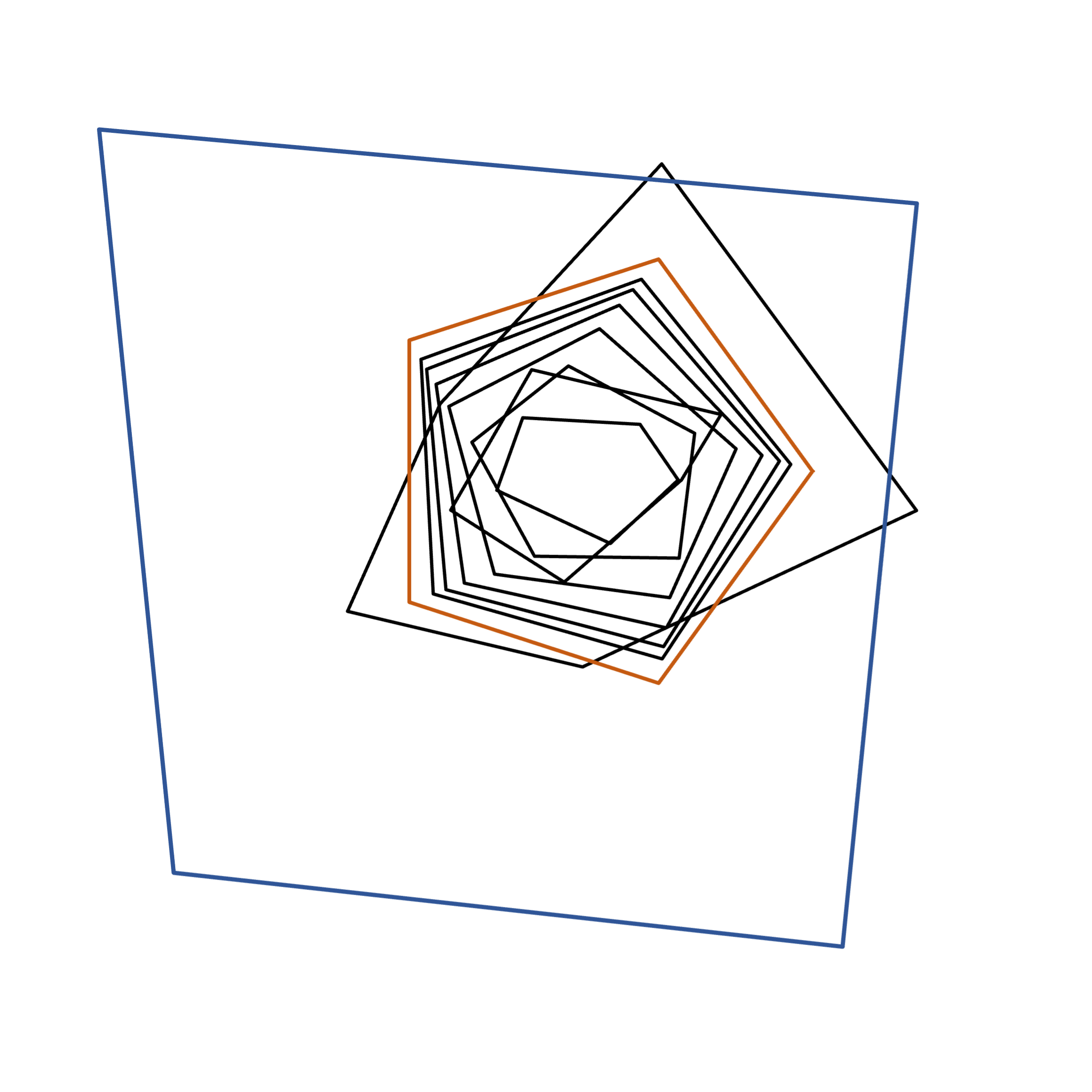}
    \caption{$m=2,\ \beta=4,\ $all $a_k=0$ }
    \label{fig:Yau_2_m=2_diff}
  \end{subfigure}
  \hfill
  \begin{subfigure}[b]{0.31\textwidth}
    \includegraphics[width=\textwidth]{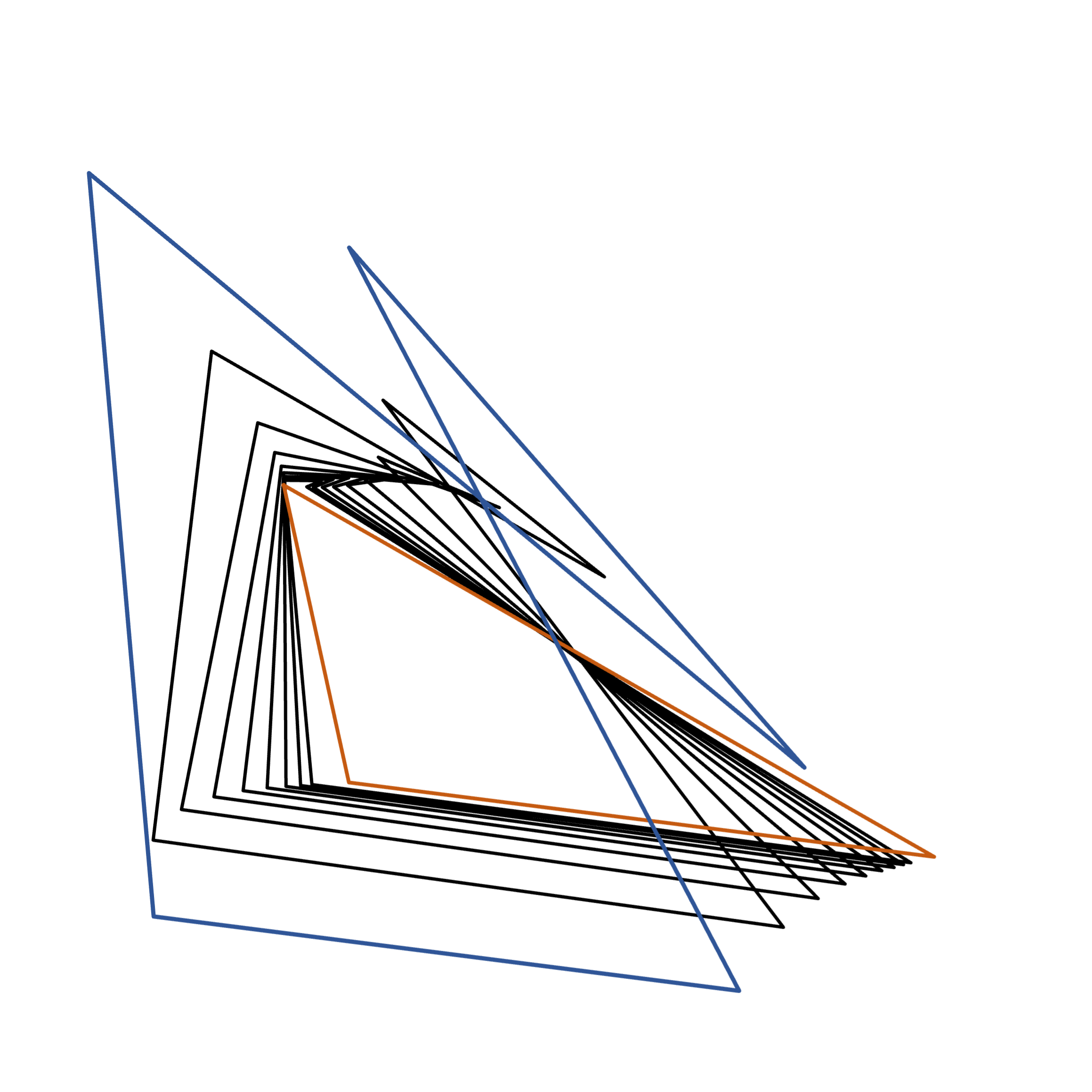}
    \caption{$m=1,\ \beta=4,\ $all $a_k = 0$}
    \label{fig:Yau_2_m=1_triangle_dup}
  \end{subfigure}
  \hfill
   \begin{subfigure}[b]{0.31\textwidth}
    \includegraphics[width=\textwidth]{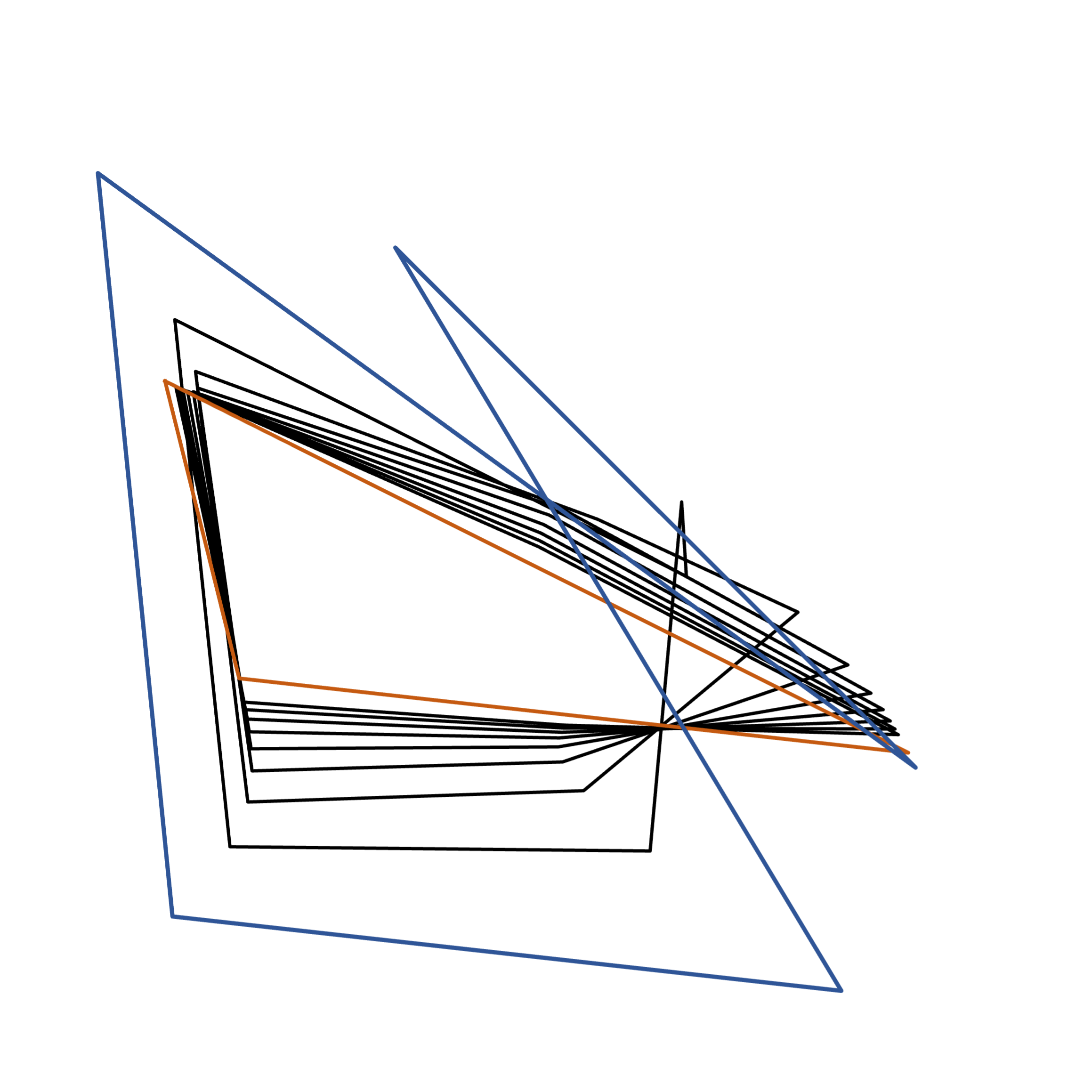}
    \caption{$m=1,\ \beta=4,\ $all $a_k = 0$ }
    \label{fig:Yau_2_m=1_triangle_line}
  \end{subfigure}
 
  \caption{Different cases of pentagons flowing under the semi-discrete Yau difference flow. In each case, selected time steps of the evolution are shown superimposed over the initial and target polygons. The initial polygon is given in blue, and the target polygon in orange. (a) depicts a quadrilateral flowing to a regular pentagon. (b) demonstrates a pentagon flowing to a triangle by duplicating excess vertices for the target polygon (c) demonstrates a pentagon flowing to a triangle by setting excess vertices to lie on the edges of the triangle.} \label{fig:Yau_cases_2}
\end{figure*}

Figure \ref{fig:Yau_cases} depicts examples of the semi-discrete Yau difference flow. In Figure \ref{fig:Yau_3_m=1_inter}, non-zero values of $a_k$ are chosen such that the polygon passes through an intermediate polygon at a fixed time before converging to the target polygon. Figure \ref{fig:Yau_cases_2} shows alternative cases for the Yau difference flow including where a polygon can be flowed to a polygon of different number of vertices by setting excess vertices along the line segments or by duplicating excess vertices.

\begin{remark}
\begin{enumerate}
    \item Hyperbolic flows that evolve one smooth curve to another are discussed in \cite{MO24}.  Parabolic flows that achieve this are described in \cite{lin2009evolving, MSW23}.  In general, some conditions on the smooth initial and target curves are needed.  For example, they might need to be strictly locally convex.  Alternatively, for curves given as radial graphs one may flow the radial graph function by the heat equation.

     \item In the case of $\beta = 0,$ there would be no exponential decay factors in the solution formula and instead of convergence to the target $\vec{Y}$ we would have oscillating about the target polygon for all time.
        \item By considering a sequence $t\rightarrow -\infty$ of initial polygons, we can construct a flow with `initial' (limiting $t\rightarrow -\infty$) polygon $\vec{X}_{-\infty}$ that passes through two states, say $\vec{X}_0$ and $\vec{X}_1$ at two distinct times before converging to the target polygon $\vec{Y}$.  The hyperbolic flow with damping \eqref{eqn:Yau_diff_eqn} allows two intermediate states $\vec{X}_0$ and $\vec{X}_1$ in such a process, as compared with a parabolic flow that would allow one intermediate state.  More intermediate states could be accommodated by using higher order flows.  Such considerations could be relevant in practical applications, for example, where a collection of robots or drones need to pass through several specific states before approach a long-term target state.

    \item The evolution equation (\ref{eqn:Yau_diff_eqn}) can flow \emph{any} initial polygon $\vec{X}_0$ with $n$ sides to \emph{any} target polygon $\vec{Y}$ with $n$ sides.  As in the smooth case and as discussed in \cite{lin2009evolving}) for example, there are also other ways of deforming one polygon to another, that do not involve a curvature flow.  In our setting we could simply take, for example, $\vec{X}:\left[ 0, 1\right] \rightarrow \mathbb{R}^p$ given by
$$\vec{X}\left( t\right) = t \vec{Y} + \left( 1-t\right) \vec{X}_0 \mbox{.}$$
  \item As with our earlier flows, one may consider ancient solutions to \eqref{eqn:Yau_diff_eqn}.  Again in comparison with the smooth case any solution at all may be extended back in time.

  \item The flow (\ref{eqn:Yau_diff_eqn}) can evolve an $n_1$-gon $\vec{X}_0$ to an $n_2$-gon $\vec{Y}$ where $n_1 \ne n_2$ are not necessarily the same.  We can either duplicate vertices or add additional vertices on the line segments of the polygon of fewer vertices. If for example $n_1 < n_2,$ then $\vec{X}_0$ is adjusted to a $n_2$-gon where some of its original vertices are duplicated (Figure \ref{fig:Yau_2_m=2_diff}). Similarly if $n_1 > n_2,$ $\vec{Y}$ can be adjusted to a $n_1$-gon with the same approach (Figure \ref{fig:Yau_2_m=1_triangle_dup}). An example of adding additional vertices along connecting line segments is given in Figure \ref{fig:Yau_2_m=1_triangle_line}. 
\end{enumerate}
   \end{remark}

\end{document}